\documentclass{amsart}
\usepackage[
  left=1.0in,    
  right=1.0in,
  top=1.0in,
  bottom=1.0in,
  includefoot    
]{geometry}
\usepackage{amsthm,amsfonts,amsmath,amssymb}
\usepackage{thmtools}
\usepackage[colorlinks=true]{hyperref}
\usepackage[nobysame]{amsrefs} 
\usepackage{cleveref}
\usepackage{tikz-cd}
\usepackage{tikz}
\usetikzlibrary{cd}
\usetikzlibrary[arrows.meta]
\usetikzlibrary{decorations.markings}
\tikzset{degil/.style={
            decoration={markings,
            mark= at position 0.5 with {
                  \node[transform shape] (tempnode) {$\backslash$};
                  }
              },
              postaction={decorate}
}
}

\usepackage{lineno}

\DeclareMathOperator{\ann}{ann}
\DeclareMathOperator{\lann}{\ell.ann}
\DeclareMathOperator{\rann}{r.ann}

\DeclareMathOperator{\morp}{End}
\DeclareMathOperator{\Hom}{Hom}
\newcommand{\homo}{\Hom}
\renewcommand{\hom}{\Hom}
\DeclareMathOperator{\id}{Id}

\DeclareMathOperator{\soc}{soc}

\DeclareMathOperator{\Aut}{Aut}

\DeclareMathOperator{\gldim}{gl.dim}
\DeclareMathOperator{\projdim}{proj.dim}
\DeclareMathOperator{\injdim}{inj.dim}
\DeclareMathOperator{\Ext}{Ext}
\DeclareMathOperator{\Tor}{Tor}

\theoremstyle{plain}
\newtheorem{theorem}{Theorem}[section] 

\newtheorem{lemma}[theorem]{Lemma}
\newtheorem{proposition}[theorem]{Proposition}
\newtheorem{corollary}[theorem]{Corollary}
\newtheorem{question}[theorem]{Question}

\theoremstyle{definition}
\newtheorem{definition}[theorem]{Definition}
\declaretheorem[
  style=definition,
  name=Example,
  sibling=theorem,
  qed=$\diamondsuit$ 
]{example}

\theoremstyle{remark}

\declaretheorem[
  style=remark,
  name=Remark,
  sibling=theorem,
  qed=$\diamondsuit$ 
]{remark}

\declaretheorem[
  style=remark,
  name=Note,
  sibling=theorem,
  qed=$\diamondsuit$ 
]{note}

\begin{document}


\pagenumbering{arabic}

\title[$p$-quasi-Frobenius Rings]{Endochain Conditions on Modules and $p$-Quasi-Frobenius Rings}

\author[Gera]{Theophilus Gera} 
\address{(Gera) Department of Mathematics, Sardar Vallabhbhai National Institute of Technology, Surat, Gujarat, India-395007}
\email{geratheophilus@gmail.com}
\thanks{The first author would like to thank the JNCASR Summer Research Fellowship Programme 2023 and the Department of Education, Government of India, for the financial assistance.}

\author[Sharma]{Amit Sharma} 
\address{(Sharma) Department of Mathematics, Sardar Vallabhbhai National Institute of Technology, Surat, Gujarat, India-395007}
\email{amitsharma@amhd.svnit.ac.in}

\author[Patel]{Manoj Kumar Patel} 
\address{(Patel) Department of Science and Humanities, National Institute of Technology Nagaland, Dimapur, Nagaland, India-797103}
\email{mkpitb@gmail.com}

\author[Kanwar]{Pramod Kanwar} 
\address{(Kanwar) Department of Mathematics, Ohio University, Zanesville, OH 43701 USA}
\email{kanwar@ohio.edu}

\date{\today}

\begin{abstract}
    Analogous to quasi-Frobenius rings, dual rings, and Johns rings, the classes of $p$-quasi-Frobenius rings, $p$-dual rings, $p$-Johns rings are introduced and used to discuss and settle principal analogues of the Faith–Menal, Faith, and FGF conjectures. A chain symmetry theorem establishes that for quasi-injective semi-projective modules with semiprime endomorphism ring, endonoetherianity and endoartinianity are equivalent and force the endomorphism ring to be semisimple artinian. Within this unified setting, it is shown that (1) two-sided \(p\)-Johns rings are \(p\)-quasi-Frobenius, giving a principal Faith–Menal analogue, with a counterexample showing that the one-sided hypothesis is insufficient; (2) semiprimary principal right ideal rings that are right injective and right Kasch are \(p\)-quasi-Frobenius—this affirmatively answers a question of Faith (Q2) in the principal two-sided ideal ring case and settles the principal Faith conjecture, with a counterexample showing the principal right ideal ring to be indispensable; and (3) right \(p\)-Johns rings that are left CF and left perfect are \(p\)-quasi-Frobenius, yielding a principal FGF analogue, with a counterexample demonstrating that the hypotheses are essential. Wherever needed, examples are provided to show that various classes introduced are different from the existing classes in the literature.
\end{abstract}

\subjclass[2020]{Primary: 16L60, 16D50; Secondary: 13D07, 16P20, 16L30, 16P40}

\keywords{$p$-quasi-Frobenius rings; $p$-Johns rings; Faith's conjecture; Faith–Menal conjecture; FGF conjecture}

\maketitle

\section{Introduction}

Quasi-Frobenius rings, introduced by Nakayama as ring-theoretic analogues of Frobenius algebras, form an important class of rings that occupy a prominent position in ring theory. Due to their classical finiteness as well as homological properties, quasi-Frobenius rings have drawn the attention of researchers working in ring and module theory as well as those working in applied areas such as coding theory. The special attention given by researchers to three celebrated open problems related to quasi-Frobenius rings, namely, Faith's conjecture, the Faith-Menal conjecture, and the FGF conjecture, attests to the depth of the theory. Despite substantial partial results in the literature by Ara et al. \cite{ara2000look}, Haim et al. \cite{haim2012two}, Iovanov \cite{iovanov2015semiprimary}, Koike and Oshiro \cite{koike2016recent}, Pardo and Asensio \citelist{\cite{pardo1997essential} \cite{pardo1997embeddings}}, among others, these conjectures remain open in full generality. They share a common structural question: under what conditions does one-sided injectivity or one-sided annihilator force the full quasi-Frobenius symmetry? (see \cite{faith2002when} for a detailed survey). \\

Over the years researchers have also introduced generalizations of injectivity, such as almost injectivity, weakly injectivity, etc. by requiring different extension conditions on homomorphisms and also generalizations by focusing on certain particular classes of ideals instead of considering the lattice of all ideals. Nicholson and Yousif (\cite{nicholson1995principally}) introduced the notion of principally injective rings (referred to as \emph{p}-injective rings for simplicity and consistency) by considering the class of principal ideals. Researchers have also explored chain conditions based on principal ideals. It is, therefore, natural to explore the impact of restricting the focus to the class principal ideals on the theory of quasi-Frobenius rings. \\

In this article, we continue this direction of research by introducing a unified framework, based on \emph{endochain conditions} and \emph{$p$-injectivity}. Our approach replaces two classical hypotheses simultaneously: artinianity (noetherianity) is weakened to endoartinianity (resp. endonoetherianity) and self-injectivity is weakened to $p$-injectivity. We introduce the notion of \emph{$p$-quasi-Frobenius rings} (see Definition \ref{def:p-quasi-Frobenius}) and \emph{right $p$-Johns rings} (\emph{left $p$-Johns rings}) (see Definition~\ref{def:p-Johns}) to provide this unified framework. Within this setting we study the principal analogous of the three classical conjectures. These analogous are obtained by replacing the classical hypotheses with their principal and endomorphism-theoretic counterparts, and asking whether the resulting structure is \( p \)-quasi-Frobenius. Here the prefix ``\( p \)'' encodes the restriction to principal ideals.\\

The choice of these definitions is not arbitrary but is forced by the desire to isolate the minimal hypotheses under which the quasi-Frobenius theory remains tractable. Replacing artinianity by endoartinianity is natural because endoartinianity is the weakest chain condition that, together with semi-projectivity, still guarantees perfectness of the endomorphism ring (see Theorem~\ref{thm:endoartinian+semi-projective=>perfect}) and interacts meaningfully with semi-projectivity. Similarly, replacing self-injectivity by \( p \)-injectivity is natural because \( p \)-injectivity is the weakest injectivity condition that still yields the double annihilator identity for principal ideals, which is precisely the level at which the classical duality theory operates.

\begin{enumerate}

    \item \textbf{Principal analogue of Faith–Menal conjecture:} \emph{Is every strongly right \( p \)-Johns ring \( p \)-quasi-Frobenius?}
    
    \noindent We show that two-sided \( p \)-Johns rings are \( p \)-quasi-Frobenius (see Theorem~\ref{thm:pJohns-pQF}). In the principal two-sided ideal setting, this coincides with the classical quasi-Frobenius property (see Corollary~\ref{cor:PRIR-QF-pQF-pJohns}). We also show that one-sided \( p \)-Johns rings are insufficient and that strong \( p \)-duality alone does not force \( p \)-quasi-Frobenius (see Example \ref{ex:strongly-pdual-not-pQF}). Consequently, because one-sided \(p\)-Johns conditions are not sufficient for \(p\)-quasi-Frobenius, there is no advantage in strengthening it to strongly \(p\)-Johns condition for this implication.

    \item \textbf{Principal analogue of Faith conjecture:} \emph{Is every semiprimary right $p$-injective ring \( p \)-quasi-Frobenius?}
    
    \noindent We show that this analogue fails in full generality, that is, semiprimary right $p$-injective rings need not be \( p \)-quasi-Frobenius (see Example~\ref{ex:semiprimary_PRIR}). However, under the additional hypotheses that \( R \) is a principal right ideal ring and right Kasch, we prove that \( R \) is \( p \)-quasi-Frobenius. If \( R \) is also a principal left ideal ring, then \( R \) is quasi-Frobenius (see Theorem~\ref{thm:semiprimary_PRIR}). This affirmatively answers a question of Faith (Q2) in the two-sided principal ideal ring case (see Corollary~\ref{cor:pir_kasch_qf}) and provides a complete resolution of the principal analogue.

    \item \textbf{Principal analogue of FGF conjecture:} \emph{Is every right CF ring \( p \)-quasi-Frobenius?}
    
    \noindent We prove that right \( p \)-Johns rings that are left CF and left perfect are \( p \)-quasi-Frobenius (see Theorem~\ref{thm:p-QF}). More generally, we prove that right CF rings satisfying coherence, AIP, and pure-extending are right perfect (see Theorem~\ref{thm:CF-PCS}). Boundary examples show that left CF alone is insufficient (see Note \ref{note:leftCF_not_p-QF}), and that \( p \)-quasi-Frobenius rings need not be CF (see Remark~\ref{rem:non-CF}). 
\end{enumerate}

We also provide an example of a $p$-quasi-Frobenius ring which is not quasi-Frobenius showing that the class of $p$-quasi-Frobenius rings is strictly larger than that of quasi-Frobenius rings. We use our example to explore how far this new class is from the class of quasi-Frobenius rings and discuss Ext and Tor groups associated with our example. Among other things, we show that classical annihilator identities do not transfer verbatim to the endomorphism-theoretic setting. For example, the Ikeda--Nakayama lemma yields equality $\lann(A \cap B) = \lann(A) + \lann(B)$ for ideals in self-injective rings; we show that for images of endomorphisms the corresponding annihilator relations hold only up to inclusion (see Proposition~\ref{prop:annihilator-endomorphisms-revised}), revealing a fundamental asymmetry between ideal-theoretic and endomorphism-theoretic behavior that necessitates new techniques. For semi-projective modules, it is shown that endoartinianity alone forces right perfectness of the endomorphism ring (see Theorem~\ref{thm:endoartinian+semi-projective=>perfect}) and the dual statement holds for quasi-injective modules under endonoetherianity (see Theorem~\ref{thm:quasi-injective+endonoetherian=>perfect}). Under the additional hypothesis that the endomorphism ring is semiprime, ascending and descending endochain conditions become equivalent, and both force semisimple artinianity (see Corollary~\ref{cor:chain-symmetry}). We also discuss Morita invariance and establish that endonoetherianity is preserved under Morita equivalence precisely when the principal annihilators correspond under the equivalence functor (see Theorem~\ref{thm:morita-princ-ann-acc}).

\section{Notation and Preliminaries}

Throughout, unless otherwise stated, all rings are associative with unity and all modules are right modules. For a ring \(R\), \(J(R)\) denotes the Jacobson radical of \(R\), and \(\soc(R_R)\) (\(\soc({}_RR)\)) denotes the right (resp. left) socle of \(R\). For a right \(R\)-module \(M\), we write \(S=\morp_R(M)\) for its endomorphism ring. For a subset \(X\subseteq R\), \(\lann(X)\) and \(\rann(X)\) denote the left and right annihilators of \(X\), respectively. If \(X=\{a\}\), we abbreviate \(\lann(\{a\})\) and \(\rann(\{a\})\) by \(\lann(a)\) and \(\rann(a)\), respectively. The field with two elements (that is, $\mathbb{F}_2$) is denoted by \(\Bbbk\).

A right \(R\)-module \(M\) is \emph{endonoetherian} (\emph{endoartinian}) if the kernels (resp. images) of endomorphisms satisfy the ascending (resp. descending) chain condition. These conditions were introduced by Kaidi and Campos \cite{kaidimodules} and further developed in \citelist{\cite{gouaid2020endo-noetherian} \cite{gera2026endoartinian}}; they refine classical finiteness by shifting attention from the lattice of submodules to the internal structure of the endomorphism ring. For the regular module \(R_R\), endonoetherianity is equivalent to the ascending chain condition on principal right annihilators, and right endoartinianity is equivalent to left perfectness.

A ring \(R\) is \emph{quasi-Frobenius} if it is right noetherian and right self-injective, equivalently, if it is two-sided artinian and two-sided self-injective (see \cite{lam1999lectures}*{Theorem 15.1}).  A right noetherian ring in which every right ideal is a annihilator ideal is called a right \emph{Johns} ring. A right Johns ring $R$ is a \emph{strongly right Johns} ring if every matrix ring $M_n(R)$ is also a right Johns ring. A ring \(R\) is right \emph{CF} (\emph{FGF}) if every cyclic (resp. finitely generated) right \(R\)-module embeds in a free module. A right FGF ring is right CF; the converse is not true (see \cite{nicholson2003quasi}*{Example 7.3}). A ring \(R\) is right \emph{AIP} if the right annihilator of every ideal of \(R\) is pure as a right ideal. Equivalently, \(R\) is right AIP if \(R/\rann(I)\) is flat for every ideal \(I\) of \(R\). $R$ is right \emph{Kasch} if every simple right $R$-module is isomorphic to a right ideal of $R$, that is, every simple right $R$-module embeds in the right module $R_R$. $R$ is a right Kasch ring if and only if every maximal ideal of $R$ has a non-zero left annihilator. A right \(R\)-module \(M\) is called \emph{reflexive} if the natural evaluation map $M \to \hom_R(\hom_R(M, R), R)$ is an isomorphism. A right $R$-module $M$ is \emph{finitely presented} if there exists an exact sequence $0 \to R^m \to R^n \to M \to 0$ with $m, n \in \mathbb{N}$. A right $R$-module $M$ is \emph{coherent} if every finitely generated submodule of $M$ is finitely presented. $R$ is coherent if $R_R$ is coherent.

Three of the classical problems on quasi-Frobenius rings, due to Faith, are:

\begin{itemize}
    \item \emph{Faith's conjecture}: every semiprimary right self-injective ring is quasi-Frobenius \cite{faith1976algebra}.
    \item \emph{Faith-Menal conjecture}: every strongly right Johns ring is quasi-Frobenius \cite{faith1994structure}.
    \item \emph{FGF (CF) conjecture}: every right FGF (CF) ring is quasi-Frobenius (resp. artinian) \cite{faith1982embedding}.
\end{itemize}

A right \(R\)-module \(M\) is called \emph{principally injective} if for every \(a\in R\), every \(R\)-homomorphism \(f:aR\to M\) extends to an \(R\)-homomorphism \(R_R\to M\). A ring \(R\) is principally injective  if \(R\) is principally injective as a right module over itself. Principally injective modules are also called divisible modules in the literature (see \cite{lam1999lectures}*{Proposition 3.17}). Henceforth, for simplicity and consistency, we use \emph{$p$-injective} for principally injective. A module \(M\) is \emph{quasi-injective} if every homomorphism from a submodule of \(M\) into \(M\) extends to an endomorphism of \(M\). $M$ is \emph{pure extending} if every pure submodule is essential in a direct summand. $M$ is \emph{semi-projective} if $Sf=\homo_R(M,f(M))$ holds for all $f\in S$. A ring $R$ is \emph{right coherent} if every finitely generated submodule of a finitely presented module is finitely presented.

For all other notation and definitions, we refer the reader to \citelist{\cite{lam1999lectures} \cite{lam2001firstcourse} \cite{wisbauer1991foundations}}.

\section{Endochain Conditions on Semi-projective and Quasi-injective Modules} \label{section2}

Recall that a right $R$-module $M$ is called \emph{endonoetherian} (resp. \emph{endoartinian}) if the kernels (resp. images) of endomorphisms of $M$ satisfy the ascending (resp. descending) chain condition. The following characterization, due to Kaidi and Campos
\cite{kaidimodules}, links endonoetherian and endoartinian properties
of a ring to chain conditions on annihilators ideals and serves as
a basis for the structural results developed in this section.
 
\begin{proposition}[\cite{kaidimodules}*{Proposition 1.16}] \label{prop:endo_noether_artin} 
Let \(R\) be a ring. Then the following statements hold.
\begin{enumerate}
    \item \(R\) is right endonoetherian if and only if it satisfies the
          ascending chain condition (ACC) on principal right
          annihilators, that is, for every ascending chain $\rann(a_1) \subseteq \rann(a_2) \subseteq \cdots,$ there exists a positive integer $n_0$ such that $\rann(a_n) = \rann(a_{n+1})$ for all $n \ge n_0$.
 
    \item \(R\) is right endoartinian if and only if it satisfies the
          descending chain condition (DCC) on principal right
          ideals, equivalently, if and only if \(R\) is left perfect.
\end{enumerate}
\end{proposition}

The following relations may be viewed as an analogue of the Ikeda-Nakayama Lemma. The classical lemma establishes, for a self-injective ring $R$, the equality $\lann(A \cap B) = \lann(A) + \lann(B)$ holds for all ideals $A, B$ of $R$ \cite{ikeda1954on}. For arbitrary rings and endomorphisms of the ring, we have the following.
 
\begin{proposition} \label{prop:annihilator-endomorphisms-revised}
For $f, g \in \morp(R_R)$ \textup{(}resp.\ $f, g \in \morp({_R}R)$\textup{)}, the following statements hold.
\begin{enumerate}
  \item For any two right ideals $T, T'$ of $R$,
        \begin{align*}
          g\!\left(\lann(T + T')\right) \subseteq g\!\left(\lann(T)\right) + g\!\left(\lann(T')\right).
        \end{align*}
        Moreover, if $g \in \Aut(R_R)$, then
        $g\!\left(\lann(T + T')\right) = g\!\left(\lann(T)\right) \cap
        g\!\left(\lann(T')\right)$.
 
  \item For any $f \in \morp(R_R)$ and for any finitely generated left ideal $L = \sum_{i=1}^{n} R x_i$ of $R$
        \[
          \sum_{i=1}^{n} \lann\!\left(f(\rann(x_i))\right)
          \subseteq \lann\!\left(f(\rann(L))\right).
        \]
        Moreover, if $\lann\!\left(f(\rann(x_i))\right) = R x_i$ for each
        $i$ \textup{(}point-wise double left annihilator\textup{)}, then
        $L = \sum_{i=1}^{n} R x_i \subseteq \lann\!\left(f(\rann(L))
        \right)$.
 
  \item  For any $g \in \morp(R_R)$ and for any finitely generated right ideal $T = \sum_{j=1}^{m} t_j R$ of $R$
        \[
          \sum_{j=1}^{m} \rann\!\left(g(\lann(t_j))\right)
          \subseteq \rann\!\left(g(\lann(T))\right).
        \]
        Moreover, if $\rann\!\left(g(\lann(t_j))\right) = t_j R$ for each
        $j$ \textup{(}point-wise double right annihilator\textup{)}, then
        $T = \sum_{j=1}^{m} t_j R \subseteq \rann\!\left(g(\lann(T))
        \right)$.
\end{enumerate}
\end{proposition}
 
\begin{proof}
\begin{enumerate}
  \item Since $\lann(T + T') = \lann(T) \cap \lann(T')$,
         $g\left(\lann(T + T')\right)=g\left(\lann(T)\cap \lann(T^{\prime})\right)$. Since for any map $g$ and
        subsets $X, Y$, 
        $g(X \cap Y) \subseteq g(X) \cap g(Y)$, we have $g\left(\lann(T + T')\right) \subseteq g\left(\lann(T))\cap g(\lann(T^{\prime})\right) \subseteq g\left(\lann(T))+ g(\lann(T^{\prime})\right)$. 
        If $g$ is invertible, then $g(X \cap Y) = g(X) \cap g(Y)$ for all
        subsets $X, Y$, giving the equality in the automorphism case.
 
  \item Since $\rann(L) = \bigcap_{i=1}^{n} \rann(x_i)$, we have
        $f(\rann(L)) \subseteq f(\rann(x_i))$ for each $i$. Consequently,
        $\lann\!\left(f(\rann(L))\right) \supseteq
        \lann\!\left(f(\rann(x_i))\right)$ for each $i$ and hence
        $$\lann\!\left(f(\rann(L))\right) \supseteq
        \sum_{i=1}^{n} \lann\!\left(f(\rann(x_i))\right).$$ 
        The second assertion is clear.
 
  \item The argument is similar to that used in (2).
\end{enumerate}
\end{proof}
\begin{remark}
Taking $f=\id_{R_R}$ in Proposition~\ref{prop:annihilator-endomorphisms-revised}(2), the condition $\lann(\rann(x_i))=Rx_i$ is precisely the \cite{nicholson1995principally}*{Lemma 1.1} of left $p$-injective rings. Likewise, taking $g=\id_{R_R}$ in Proposition~\ref{prop:annihilator-endomorphisms-revised}(3) yields $\rann(\lann(t_j))=t_jR$, which is the corresponding characterization of right $p$-injective rings. Thus, the hypotheses in parts (2) and (3) may be regarded as endomorphism-twisted analogous of the classical double-annihilator conditions for $p$-injectivity. Moreover, Proposition~\ref{prop:annihilator-endomorphisms-revised} shows that these point-wise endomorphism-twisted double-annihilator conditions naturally extend to containment relations for finitely generated one-sided ideals.
\end{remark}

Recall that a module $M$ is semi-projective if $Sf=\homo_R(M,f(M))$ holds for all $f\in S$. Harada proved that for an artinian semi-projective module \(M_R\), the endomorphism ring \(S=\morp_R(M)\) is semi-primary, hence perfect (\cite{harada1970semi}*{Proposition 2.4}). In our next theorem we consider the case of endoartinian semi-projective modules. 
 
\begin{theorem}\label{thm:endoartinian+semi-projective=>perfect}
    Let $M$ be an endoartinian semi-projective $R$-module. Then $S$ is left endoartinian (that is, right perfect). Moreover, if $S$ is semiprime, then $S$ is semisimple artinian.
\end{theorem}
 
\begin{proof}
    Let $Sf_1 \supseteq Sf_2 \supseteq Sf_3 \supseteq \cdots$ be a descending chain of principal left ideals of $S$. Then $f_{i+1} \in Sf_i$ for each $i$. Let $f_{i+1} = s_i f_i$ for some $s_i \in S$. Then for every $m\in M$, $f_{i+1}(m)=(s_i f_i)(m)$. Since $M$ is semi-projective, $Sf_i=\hom_R(M,f_i(M))$. Hence $f_{i+1}(m) \in f_i(M)$. It follows that $f_1(M) \supseteq f_2(M) \supseteq f_3(M) \supseteq \cdots$ is a descending chain of images of $M$. Since $M$ is endoartinian, there exists $n$ such that $f_k(M) = f_{k+1}(M)$ for all $k \geq n$.
    
    Again invoking semi-projectivity of $M$, we have $Sf_i=\hom_R(M,f_i(M))$ for each $i$. Hence for all $k \geq n$, $S f_{k}=\hom_R(M,f_k(M))=\hom_R(M,f_{k+1}(M))=Sf_{k+1}$. By Bass's Theorem~P \cite{lam2001firstcourse}*{Theorem~23.20}, $S$ is right perfect.
 
   If $S$ is semiprime, by \cite{lam2001firstcourse}*{Theorem 10.24}, $S$ is semisimple artinian. 
\end{proof}
 
We now consider the complementary case of endonoetherian quasi-injective modules. We first prove the following lemma that will be useful to prove the main result.
 
\begin{lemma}\label{lem:qp-factor}
    Let $M$ be a quasi-injective $R$-module and let $f, g \in S$. If $\ker f \subseteq \ker g$, then there exists $u \in S$ such that $g = uf$.
\end{lemma}
 
\begin{proof}
    Define $\bar{g} : f(M) \to M$ by $\bar{g}(f(m)) := g(m)$. This is well-defined since $\ker f \subseteq \ker g$. Clearly, $\bar{g}$ is an $R$-module homomorphism. Since $f(M)$ is a submodule of $M$ and $M$ is quasi-injective,  there exists $u \in S$ such that $u\vert_{f(M)} = \bar{g}$. Now for any $m \in M$: $(uf)(m) = u(f(m)) = \bar{g}(f(m)) = g(m)$. Hence $g = uf$.
\end{proof}
 
\begin{theorem}\label{thm:quasi-injective+endonoetherian=>perfect}
    Let $M$ be an endonoetherian quasi-injective $R$-module. Then $S$ is left endoartinian (that is, right perfect). Moreover, if $S$ is semiprime, then $S$ is semisimple artinian.
\end{theorem}
 
\begin{proof}
    Let $Sf_1 \supseteq Sf_2 \supseteq \cdots$ be a descending chain of principal left ideals in $S$. Then, for each $i$, there exists $t_i \in S$ such that $f_{i+1} = t_i f_i$. Consequently, $\ker f_i \subseteq \ker f_{i+1}$ for each $i$. Thus $\ker f_1 \subseteq \ker f_2 \subseteq \ker f_3 \subseteq \cdots$ is an ascending chain of submodules of $M$. Since $M$ is endonoetherian, there exists a positive integer $n_0$ such that $\ker f_n = \ker f_{n+1}$ for all $n \ge n_0$. Fix $k \ge n_0$. By Lemma~\ref{lem:qp-factor},  there exists $u \in S$ such that $f_k = u f_{k+1}$. Hence $S f_k \subseteq S f_{k+1}$. It follows that $S f_k = S f_{k+1}$. Hence $S f_n = S f_{n+1}$ for all $n \ge n_0$. By Bass's Theorem~P \cite{lam2001firstcourse}*{Theorem~23.20}, $S$ is right perfect.
 
    If $S$ is semiprime, by \cite{lam2001firstcourse}*{Theorem 10.24}, $S$ is semisimple artinian. 
\end{proof}

Combining Theorem~\ref{thm:endoartinian+semi-projective=>perfect} and Theorem~\ref{thm:quasi-injective+endonoetherian=>perfect}, we get the following corollary showing that under semiprimeness, ascending and descending chain conditions on endoimages are equivalent. 
 
\begin{corollary}\label{cor:chain-symmetry}
Let $M$ be a quasi-injective and semi-projective right $R$-module and assume that $S$ is semiprime. Then the following statements are equivalent:
\begin{enumerate}
    \item $M$ is endonoetherian;
    \item $M$ is endoartinian;
    \item $S$ is semisimple artinian.
\end{enumerate}
\end{corollary}
 
\begin{remark}\label{rem:chain_symmetry_semiprime}
The semiprime hypothesis in Corollary~\ref{cor:chain-symmetry} is essential. Let $R = \Bbbk[x]/(x^2)$ be the algebra of dual numbers over a field $\Bbbk$, and $M = R_R$. Then $S = \morp_R(R_R) \cong R$ is a local ring with $J(S) = \Bbbk x \neq 0$, hence not semiprime. Since $R$ is a symmetric Frobenius algebra, $R_R$ is quasi-injective. Although $M$ is endonoetherian (since $R$ is artinian), $S$ is not semisimple artinian. Thus the implication $(1) \Rightarrow (3)$ in Corollary~\ref{cor:chain-symmetry} fails when the semiprime hypothesis is dropped.
\end{remark}

\begin{note}\label{note:lam-connection}
    When \(M=R_R\), every endomorphism of \(M\) is given by left multiplication by \(f(1)\); hence the endomorphic images of \(M\) are precisely the principal right ideals of \(R\). Thus Corollary~\ref{cor:chain-symmetry} may be viewed as a module-theoretic generalization of \cite{lam2001firstcourse}*{Theorem~10.24}. The semiprimeness hypothesis transfers unchanged from \(R\) to \(S\), while the annihilator-theoretic ingredient in the ring case is replaced by the quasi-injectivity of \(M\), which provides the endomorphism factorization in Lemma~\ref{lem:qp-factor}.
\end{note}
 
These results indicate that endoartinianity, together with quasi-injectivity conditions, already captures much of the rigidity traditionally associated with quasi-Frobenius rings, leading naturally to the notion of $p$-quasi-Frobenius rings.

\section{$p$-Quasi-Frobenius Rings} \label{section3}

Recall that a ring \( R \) is said to be \emph{quasi-Frobenius} if it is two-sided artinian and two-sided self-injective. Motivated by this, we define a class of rings where artinianity is relaxed to endoartinianity (that is, perfectness) and self-injectivity is weakened to $p$-injectivity.

\begin{definition} \label{def:p-quasi-Frobenius}
    A ring \( R \) is called \emph{\( p \)-quasi-Frobenius} if it is two-sided endoartinian (that is, perfect) and two-sided $p$-injective.
\end{definition}

The prefix $``p"$, in \emph{\( p \)-quasi-Frobenius}, is for \textit{principal} emphasizing the role of principal ideals and $p$-injectivity. Clearly, the class of $p$-quasi-Frobenius rings lies between quasi-Frobenius rings and $p$-injective rings. Indeed,
\begin{align*}
    \text{quasi-Frobenius} \Rightarrow p\text{-quasi-Frobenius} \Rightarrow \text{left and right $p$-injective}
\end{align*}

Since a quasi-Frobenius ring is two-sided artinian and two-sided self-injective and hence, in particular, two-sided endoartinian and two-sided $p$-injective, the first implication is immediate. The second implication follows directly from the definition of $p$-quasi-Frobenius rings.

A less obvious, but important, consequence is that every $p$-quasi-Frobenius ring is also left and right Kasch. Indeed, a $p$-quasi-Frobenius ring is two‑sided perfect, so by Bass’s theorem its left socle is essential as a left ideal and its right socle is essential as a right ideal \cite{bass1960finitistic}. Moreover, $R$ is semiperfect and two‑sided $p$-injective. Applying \cite{nicholson2003quasi}*{Proposition~5.20}, we get $R$ is right Kasch. The left‑hand version follows symmetrically. Thus, we have 
\begin{align*}
    \text{quasi-Frobenius} \Rightarrow p\text{-quasi-Frobenius} \Rightarrow \text{left and right Kasch}
\end{align*}

The reverse implications fail, in general. First, for an example of a Kasch ring that is not \(p\)-quasi-Frobenius, consider the triangular ring $R=\begin{pmatrix}
    A & B\\
    0 & A
    \end{pmatrix}$, 
where \(A=\mathbb{Z}/4\mathbb{Z}\) and \(B=2\mathbb{Z}/4\mathbb{Z}\) is the ideal of \(A\). This ring is Kasch but fails to be $p$-injective on the left, and hence is not \(p\)-quasi-Frobenius (see \cite{lam1999lectures}*{pp.~79--80}).

Next, for an example of a \(p\)-quasi-Frobenius ring that is not quasi-Frobenius, consider the ring $R=\Bbbk [x_1,x_2,x_3,\cdots]\big/\big(x_i^3,\;x_ix_j\ (i\neq j),\;x_i^2-x_j^2 \textnormal{ for all }i,j\ge 1\big)$. Let $\mathfrak m:=x_1^2\neq 0$. Note that $\mathfrak m=x_i^2$ for all $i$. Then $J = J(R)=\left(\bigoplus_{i\ge1}\Bbbk x_i\right)\oplus \Bbbk \mathfrak m$ is the unique maximal ideal of \(R\). By \cite{camillo2000ikeda}*{Example~6}, \(R\) is quasi-continuous but not Ikeda--Nakayama. In particular, \(R\) is not self-injective and hence not quasi-Frobenius. Since \(J^3=0\), \(R\) is perfect; being commutative, it is two-sided endoartinian. Thus it remains to verify that \(R\) is $p$-injective. Since \(R\) is commutative, by \cite{nicholson1995principally}*{Lemma 1.1}, it suffices to prove that \(\ann_R(\ann_R(a))=Ra\) for every \(a\in R\).

\begin{enumerate}
    \item If \(a\notin J\), then \(a\) is a unit, so \(\ann_R(a)=0\) and hence \(\ann_R(\ann_R(a))=R=Ra\).

    \item If \(a=\mathfrak m\), then \(\ann_R(a)=J\), since \(x_i\mathfrak m=0\). Moreover, \(\ann_R(J)=R\mathfrak m\), because an element $r=\alpha+\sum_i\beta_i x_i+\gamma\mathfrak m$ annihilates every \(x_i\) if and only if \(\alpha=\beta_i=0\) for all \(i\). Thus $\ann_R(\ann_R(a))=\ann_R(J)=R\mathfrak m=Ra.$

    \item If $a=\sum_{k\in F}x_k+\gamma\mathfrak m,$ where \(F\) is a nonempty finite subset of \(\mathbb N\), then a direct computation shows that $\ann_R(a)=\left\{\sum_i\beta_i x_i+\delta\mathfrak m:\sum_{k\in F}\beta_k=0\right\}$. Hence $\ann_R(\ann_R(a)) =\left\{\lambda\sum_{k\in F}x_k+\mu\mathfrak m:\lambda,\mu\in\Bbbk \right\}$. Since $Ra=\left\{\lambda\sum_{k\in F}x_k+\mu\mathfrak m:\lambda,\mu\in\Bbbk \right\}$, it follows that \(\ann_R(\ann_R(a))=Ra\).
\end{enumerate}
Therefore, \(R\) is $p$-injective. Together with the perfectness noted above, it follows that \(R\) is \(p\)-quasi-Frobenius. 

\begin{remark} \label{rem:p-QF_ex_remark}
The example of a \(p\)-quasi-Frobenius ring $R$ which is not quasi-Frobenius, given above, illustrates several structural phenomena. As observed above, for this ring, $J=J(R)=\bigoplus_{i=1}^{\infty}\Bbbk x_i\oplus\Bbbk \mathfrak m$, $J^2=\Bbbk \mathfrak m,$ and $J^3=0.$

\begin{enumerate}

\item \emph{Infinite Loewy width.} The Loewy series $0\subset J^2\subset J\subset R$ has length $3$. However,  $J/J^2\cong\bigoplus_{i=1}^{\infty}\Bbbk x_i$ is an infinite direct sum of simple modules. Hence the middle Loewy layer has infinite width. This shows that $p$-quasi-Frobenius rings need not have the rigid Loewy structure typical of quasi-Frobenius rings.

\item \emph{Quotients of $p$-quasi-Frobenius rings need not be $p$-quasi-Frobenius.} Let $\overline{R}=R/(\mathfrak{m})$. Then $\overline{x}_i\overline{x}_j=0$ for all $i$ and $j$. So, $J(\overline{R})=\bigoplus_{i=1}^{\infty}\Bbbk \overline{x}_i$, and $J(\overline{R})^2=0$.

For any element $r=c+\sum_{i}d_i\overline{x}_i\in\overline{R}$, we have $\overline{x}_1 r=c\,\overline{x}_1$. Hence $\overline{x}_1 r=0$ if and only if $c=0$, so $\ann_{\overline{R}}(\overline{x}_1)=J(\overline{R})$.

If $s=c+\sum e_i\overline{x}_i\in\overline{R}$ annihilates $J(\overline{R})$, then $s\overline{x}_j=c\overline{x}_j=0$ for all $j$, forcing $c=0$. Thus $\ann_{\overline{R}}(J(\overline{R}))=J(\overline{R})$. However, $\overline{R}\,\overline{x}_1=\Bbbk \overline{x}_1 \neq J(\overline{R})$. Therefore, $\ann_{\overline{R}}\big(\ann_{\overline{R}}(\overline{x}_1)\big) \neq \overline{R}\,\overline{x}_1$, so $\overline{R}$ is not left $p$-injective. Hence $\overline{R}$ is not $p$-quasi-Frobenius.

\item \emph{Subrings of $p$-quasi-Frobenius rings need not be $p$-quasi-Frobenius.} Let $a=x_1+x_2$ and $b=x_3+x_4$ in \(R\). Then $a^2=(x_1+x_2)^2=x_1^2+x_2^2=\mathfrak{m}+\mathfrak{m}=0$, $b^2=(x_3+x_4)^2=x_3^2+x_4^2=\mathfrak{m}+\mathfrak{m}=0$, and $ab=(x_1+x_2)(x_3+x_4) =x_1x_3+x_1x_4+x_2x_3+x_2x_4=0$. Thus, the subring $S=\Bbbk [a,b]\subseteq R$ is isomorphic to $\Bbbk [a,b]/(a^2,b^2,ab)$ (which is not Gorenstein, not quasi-Frobenius). Its Jacobson radical is $J(S)=\operatorname{span}\{a,b\},$ with \(J(S)^2=0\). Hence \(S\) is local and perfect. For \(a\in S\), $\ann_S(a)=\{s\in S\mid sa=0\} =J(S)=\operatorname{span}\{a,b\}.$ Then $\ann_S(\ann_S(a)) =\ann_S(J(S)) =\operatorname{span}\{a,b\} =J(S)$. However, $Sa=\{0,a\}$. So, $\ann_S(\ann_S(a))\neq Sa.$ Therefore, \(S\) is not $p$-injective and consequently, \(S\) is not \(p\)-quasi-Frobenius.

\item \emph{Failure of quasi-Frobenius duality.} The regular module \(R_R\) is projective. We show that it is not injective. Define $f:J\to R$ by $f\!\left(\sum_{i}a_ix_i+c\mathfrak m\right)=\sum_{i}a_i\mathfrak m.$ Since \(\{x_i,\mathfrak m\}\) is a \(\Bbbk \)-basis of \(J\), this is well defined. Using \(x_ix_j=0\) for \(i\neq j\), \(x_i^2=\mathfrak m\), and \(\mathfrak m x_i=\mathfrak m^2=0\), one checks directly that \(f\) is an \(R\)-homomorphism.

Suppose \(f\) extended to an \(R\)-homomorphism \(R_R\to R_R\). Then there exists \(t\in R\) such that $f(y)=ty$ for all $y\in J$. In particular, $\mathfrak m=f(x_i)=tx_i$ for all $i\geq 1$. Every element of \(R\) has a unique expression $t=\alpha+\sum_{i\ge1}\beta_i x_i+\gamma\mathfrak m,$ where \(\alpha,\gamma,\beta_i\in\Bbbk \) and only finitely many \(\beta_i\) are nonzero. Choose \(i\) with \(\beta_i=0\). Then $tx_i=\alpha x_i+\beta_i\mathfrak m=\alpha x_i.$ Thus \(\mathfrak m=\alpha x_i\), which is impossible because \(\alpha x_i\) is either \(0\) or \(x_i\), and neither equals \(\mathfrak m\). Hence no such \(t\) exists, so \(f\) does not extend to \(R\). Therefore \(R_R\) is not injective.

Consequently, in this \(p\)-quasi-Frobenius ring, the regular module is projective but not injective. Thus the classical quasi-Frobenius duality between projective and injective modules fails for this example.
\end{enumerate}
\end{remark}

In the following note, we compute the perfect closure of the $p$-quasi-Frobenius ring $R$ along with the associated Ext and Tor groups, thereby exhibiting explicit homological obstructions to finite injective and  projective dimensions.

\begin{remark} \label{rem:infinite-projective-injective-dimension}
We retain the notation of Remark \ref{rem:p-QF_ex_remark}. 

\emph{The perfect closure.} We examine the images of $R$ under iterates of the Frobenius map \(F\) on $R$. Recall that, since \(R\) is an \(\Bbbk \)-algebra, the Frobenius map on $R$ is \(F(r)=r^2\) for all $r \in R$. We have $F(R)=\Bbbk \oplus \Bbbk \mathfrak{m}$ and $F^2(R)=F(\Bbbk \oplus \Bbbk \mathfrak{m})$. Since \(\mathfrak{m}^2=0\), $F^2(R)=\Bbbk$. It follows that for all \(n\ge 2\), \(F^n(R)=\Bbbk\) and the Frobenius acts as the identity on \(\Bbbk\). Therefore, the direct limit $R^\infty:=\varinjlim\bigl(R\xrightarrow{F}R\xrightarrow{F}R\xrightarrow{F}\cdots\bigr)$ stabilises after two steps, yielding $R^\infty\cong \Bbbk$. Thus, for any \(R\)-module \(N\), $\Ext^i_R(R^\infty,N)\cong \Ext^i_R(\Bbbk,N)$ and $\Tor^R_i(R^\infty,N)\cong \Tor^R_i(\Bbbk,N).$

    \begin{enumerate} 
\item \emph{Failure of finite injective dimension.} We first compute \(\Ext^1_R(\Bbbk,R)\). Applying \(\hom_R(-,R)\) functor to the short exact sequence
\[
0\longrightarrow J\longrightarrow R\longrightarrow \Bbbk\longrightarrow 0.
\]
we obtain the exact sequence
\[
0\to \hom_R(\Bbbk,R)\to \hom_R(R,R)
\xrightarrow{\alpha} \hom_R(J,R)\to \Ext^1_R(\Bbbk,R)\to 0.
\]
Here $\hom_R(\Bbbk,R)\cong \ann_R(J)=J^2=\Bbbk \mathfrak{m},$ and \(\hom_R(R,R)\cong R\) via $r\mapsto (s\mapsto rs)$. The map \(\alpha\) is induced by the inclusion \(J\hookrightarrow R\); explicitly, it sends a homomorphism \(\varphi:R\to R\) to its restriction \(\varphi|_J:J\to R\). 

To describe the cokernel, note that every \(R\)-linear map \(f:J\to R\) is uniquely determined by a common scalar \(\delta\in \Bbbk\) and a binary sequence \((\gamma_i)_{i\in\mathbb N}\in \Bbbk^{\mathbb N}\), via $f(x_i)=\delta x_i+\gamma_i \mathfrak{m}$ and $f(\mathfrak{m})=\delta \mathfrak{m}.$ Indeed, \(x_i f(x_i)=f(\mathfrak{m})\) forces the coefficient of \(x_i\) in \(f(x_i)\) to be exactly \(\delta\) for all \(i\). The image of \(\alpha\) consists precisely of those maps with \(\delta=\lambda\) and \((\gamma_i)=(\beta_i)\) having finite support, that is, the subspace $\Bbbk\times \Bbbk^{(\mathbb N)}$ inside $\Bbbk\times \Bbbk^{\mathbb N}$. Since \(\Ext^1_R(\Bbbk,R)\) is the cokernel of \(\alpha\), we obtain 
\[
E_1:=\Ext^1_R(\Bbbk,R) \cong \frac{\Bbbk\times \Bbbk^{\mathbb N}}
{\Bbbk\times \Bbbk^{(\mathbb N)}}
\cong
\frac{\Bbbk^{\mathbb N}}{\Bbbk^{(\mathbb N)}}
\neq 0.
\]

We now prove that \(E_i:=\Ext^i_R(\Bbbk,R)\neq 0\) for every \(i\ge 1\), which will directly imply \(\injdim_R(R)=\infty\). Suppose, for contradiction, that \(E_n=0\) for some \(n\ge 2\), and choose such an \(n\) minimal. Consider the short exact sequence
\[
0\longrightarrow \Bbbk\longrightarrow J\longrightarrow \Bbbk^{(\mathbb N)}\longrightarrow 0,
\]
where the inclusion \(\Bbbk\hookrightarrow J\) sends \(1\mapsto \mathfrak{m}\), and the surjection sends \(x_i\mapsto \bar x_i\) and \(\mathfrak{m}\mapsto 0\). Applying \(\hom_R(-,R)\) and using the isomorphisms $\Ext^i_R(J,R)\cong \Ext^{i+1}_R(\Bbbk,R)=E_{i+1}$ for all $i\geq 1$, which follow from \(0\to J\to R\to \Bbbk\to 0\) since \(\Ext^i_R(R,R)=0\) for \(i>0\), together with $\Ext^i_R(\Bbbk^{(\mathbb N)},R)\cong E_i^{\mathbb N},$ we obtain, for each \(i\ge 1\), the exact sequence
\[
\cdots\to E_i^{\mathbb N}\to E_{i+1}\to E_i\to E_{i+1}^{\mathbb N}\to E_{i+2}\to E_{i+1}\to\cdots.
\]
Since \(E_1\neq 0\), suppose for contradiction that \(E_n=0\) for some \(n\ge 2\), and choose such an \(n\) minimal. Setting \(i=n-1\) in the exact sequence above yields the fragment
\[
E_{n-1}^{\mathbb N}\to E_n\to E_{n-1}\to E_n^{\mathbb N}\to E_{n+1}\to\cdots.
\]
Since \(E_n=0\), exactness at \(E_n\) forces the map $E_{n-1}^{\mathbb N}\to E_n$ to be zero and the map $E_n\to E_{n-1}$ to be injective. Exactness at \(E_{n-1}\) then forces the map $E_{n-1}\to E_n^{\mathbb N}=0$ to be injective, which is possible only if \(E_{n-1}=0\). This contradicts the minimality of \(n\). Hence no such \(n\) exists, and therefore $\Ext^i_R(\Bbbk,R)\neq 0$ for all $i\ge 1$. Consequently, \(\injdim_R(R)=\infty\), so \(R\) is not quasi-Frobenius.

\item \emph{Failure of finite projective dimension.} 
Let $T_i:=\Tor_i^R(\Bbbk,\Bbbk)$ for each \(i\ge 0\). We have the exact sequence of \(R\)-modules
\[
0\to k\to J\to \Bbbk^{(\mathbb N)}\to 0,
\]
as discussed above. Tensoring this sequence with \(\Bbbk\) over \(R\) yields, for each \(i\ge 1\), the long exact Tor sequence. To make this tractable, we use two standard isomorphisms. First, from the exact sequence
\[
0\to J\to R\to \Bbbk\to 0,
\]
we have $\Tor_i^R(\Bbbk,J)\cong \Tor_{i+1}^R(\Bbbk,\Bbbk)=T_{i+1}$ for all $i \geq 1$. Second, because Tor commutes with arbitrary direct sums, we obtain $\Tor_j^R(\Bbbk,\Bbbk^{(\mathbb N)})\cong T_j^{(\mathbb N)}$ for every \(j\). Substituting these into the long exact sequence of \(\Tor_*^R(\Bbbk,-)\) applied to
\[
0\to \Bbbk\to J\to \Bbbk^{(\mathbb N)}\to 0
\]
gives the following exact sequence, for each \(i\ge 1\):
\[
\cdots\to T_{i+1}^{(\mathbb N)}\to T_i\to T_{i+1}\to T_i^{(\mathbb N)}\to T_{i-1}\to T_i\to T_{i-1}^{(\mathbb N)}\to\cdots.
\]
We now prove that \(T_n\neq 0\) for every \(n\ge 1\). The base case is immediate: $T_1\cong J/J^2\cong \Bbbk^{(\mathbb N)}\neq 0.$ Suppose, for contradiction, that some \(T_n=0\) with \(n\ge 2\), and choose such an \(n\) minimal. Setting \(i=n-1\) in the exact sequence above isolates the fragment $0\longrightarrow T_{n-1}\longrightarrow 0$, because \(T_n=0\) implies both \(T_n\) and \(T_n^{(\mathbb N)}\) vanish. Exactness at the middle term \(T_{n-1}\) forces its kernel and image to match, which immediately yields \(T_{n-1}=0\). This contradicts the minimality of \(n\). Therefore no such \(n\) exists, and $T_i\neq 0$ for every $i\ge 1$. Since the perfect closure satisfies \(R^\infty\cong k\), we have $\Tor_i^R(R^\infty,\Bbbk)\cong \Tor_i^R(\Bbbk,\Bbbk)=T_i\neq 0$ for all $i \geq 1$.

We now conclude directly that \(\projdim_R(\Bbbk)=\infty\). Indeed, if \(\projdim_R(\Bbbk)=d<\infty\), then for every \(R\)-module \(N\) one has $\Tor_i^R(\Bbbk,N)=0$ for all  $i> d$. In particular, \(T_i=0\) for all \(i>d\), contradicting the fact that \(T_i\neq 0\) for every \(i\ge 1\). Hence $\projdim_R(\Bbbk)=\infty$, and  therefore $\gldim(R)=\infty$.

This provides the projective analogue of the injective-dimension failures above, confirming that non-noetherian \(p\)-quasi-Frobenius rings can be homologically wild on both sides of the Ext-Tor duality.
    \end{enumerate}
\end{remark}

Next, we show that, unlike quasi-Frobenius rings, \(p\)-quasi-Frobenius rings need not satisfy the usual reflexive duality. The following example exhibits a finitely generated module over the \(p\)-quasi-Frobenius ring \(R\) discussed above that is not reflexive.

\begin{example}\label{ex:explicit-nonreflexive-module}
Let $R=\Bbbk[x_1,x_2,\ldots]\big/\big(x_i^3,\;x_ix_j\ (i\neq j),\;x_i^2-x_j^2\big)$ be the \(p\)-quasi-Frobenius ring discussed above. Consider the cyclic \(R\)-module $M=R/(x_1,x_2).$ Although \(M\) is cyclic, it is infinite-dimensional, as a \(\Bbbk\)-vector space. We show that the canonical evaluation map $\delta_M:M\to M^{**}$ is not surjective.

Let \(I=(x_1,x_2)\). Since \(M\cong R/I\), we have $M^*=\hom_R(M,R)\cong \ann_R(I)=\ann_R(x_1)\cap\ann_R(x_2).$
A direct computation gives $\ann_R(x_1)=\bigoplus_{i\neq1}\Bbbk x_i\oplus\Bbbk \mathfrak m,$ and $\ann_R(x_2)=\bigoplus_{i\neq2} \Bbbk x_i\oplus\Bbbk \mathfrak m.$ Hence $M^*\cong L:=\bigoplus_{i\ge3}\Bbbk x_i\oplus\Bbbk \mathfrak m.$

Define an \(R\)-homomorphism \(\psi:L\to R\) by $\psi(x_i)=\mathfrak m$ for $i\ge 3$, and $\psi(\mathfrak m)=0$. To check \(R\)-linearity, it suffices to test the generators. Since \(x_i^2=\mathfrak m\), \(x_ix_j=0\) for \(i\neq j\), \(x_i\mathfrak m=0\), and \(\mathfrak m^2=0\), we have $x_i\psi(x_i)=x_i\mathfrak m=0=\psi(x_i^2),$ and, for \(i\neq j\), $x_i\psi(x_j)=x_i\mathfrak m=0=\psi(x_ix_j).$ Also, $x_i\psi(\mathfrak m)=0=\psi(x_i\mathfrak m).$ Thus, \(\psi\) respects the defining relations and hence is a well-defined \(R\)-module homomorphism.

Now, let $u=\alpha+\sum_{j\ge1}\beta_jx_j+\gamma\mathfrak m\in R.$ Under the identification \(M^*\cong L\), the canonical map \(\delta_M(\bar u)\in M^{**}=\hom_R(L,R)\) is given by $\delta_M(\bar u)(\ell)=\ell u$ for $\ell\in L$. In particular, for \(i\ge3\), $\delta_M(\bar u)(x_i)=x_iu=\alpha x_i+\beta_i\mathfrak m.$ Since every element of \(R\) involves only finitely many \(x_i\), the set $\{\,i\ge3:\delta_M(\bar u)(x_i)=\mathfrak m\,\}$ is finite.

However, \(\psi(x_i)=\mathfrak m\) for every \(i\ge3\). Therefore, \(\psi\) cannot lie in the image of \(\delta_M\). Hence \(\delta_M\) is not surjective, and the cyclic, hence finitely generated, module $M=R/(x_1,x_2)$ is not reflexive. 

Moreover, since \(\mathfrak{m}=x_1^2\in I\), we have \(\bar{\mathfrak{m}}=0\) in \(M\). Moreover, \(x_iM=0\) for every \(i\ge 1\), because \(x_i\bar x_j=0\) for \(i\neq j\) and \(x_i\bar x_i=\mathfrak{m}=0\). Hence \(JM=0\). It follows that \(M\) is a semisimple \(R\)-module and hence $M\cong \Bbbk^{(\Lambda)}$ for some index set \(\Lambda\) (in fact, $M\cong \Bbbk\oplus \bigoplus_{i\ge 3}\Bbbk \bar x_i$). Therefore, $\Ext^i_R(M,R) \cong \Ext^i_R(\Bbbk^{(\Lambda)},R) \cong \prod_\Lambda \Ext^i_R(\Bbbk,R) \cong \prod_\Lambda E_i \neq 0$ for every \(i\ge 1\), since each \(E_i\neq 0\). Hence \(\injdim_R(M)=\infty\) which is consistent with the non-reflexivity of \(M\) established in Remark~\ref{rem:infinite-projective-injective-dimension}(1). 
\end{example}

It is known that a module over a quasi-Frobenius ring is reflexive if and only if it is finitely generated. While Example~\ref{ex:explicit-nonreflexive-module} shows the forward implication is destroyed in the $p$-quasi-Frobenius setting, the reverse implication touches upon a long-standing open problem \emph{whether every reflexive module over a perfect S-ring (that is, a perfect Kasch ring) is necessarily finitely generated} due to Schulz (\cite{schulz1979reflexive}*{p. 535}). A positive answer to this question would prove that every perfect right injective cogenerator (right psuedo-Frobenius) ring is quasi-Frobenius, thereby resolving the one-sided formulation of Faith's conjecture. While the two-sided case is settled — a perfect ring which is a right and left cogenerator is necessarily quasi-Frobenius \cite{faith2002when}*{Corollary 2.4} — the one-sided case remains notoriously open.

Since $p$-quasi-Frobenius rings are perfect Kasch rings, they provide a rigorous testing ground for Schulz's question. This leads us to raise the following question: 

\begin{question}
    Is every reflexive module over a $p$-quasi-Frobenius ring finitely generated?
\end{question} 

For a principal two-sided ideal ring, the notions of quasi-Frobenius and $p$-quasi-Frobenius coincide. This allows us to characterize quasi-Frobenius rings directly in terms of endoartinianity (i.e., perfectness). We include this result for the sake of completeness. 

\begin{theorem} \label{thm:PIR-pQF-QF}
Let $R$ be a principal two-sided ideal ring. Then the following statements are equivalent:
\begin{enumerate}
    \item $R$ is quasi-Frobenius;
    \item $R$ is $p$-quasi-Frobenius.
\end{enumerate}
\end{theorem}

Since factor rings of a principal two-sided ideal ring are again principal
two-sided ideal rings, and since the artinian property is inherited by factor rings, the preceding theorem yields the following.

\begin{corollary}\label{cor:factor-pqf}
Let $R$ be an artinian principal two-sided ideal ring. Then every factor ring of $R$ is $p$-quasi-Frobenius. Equivalently, within this class, the property of being $p$-quasi-Frobenius is preserved under factor rings.
\end{corollary}

The following result is straightforward due to Rutter \cite{rutter1969characterizations}*{Theorem 2}.

\begin{corollary}\label{cor:cyclic-reflexive-pqf-QF}
    Let $R$ be a $p$-quasi-Frobenius ring. If every cyclic left $R$-module is reflexive, then $R$ is quasi-Frobenius.
\end{corollary}

Finally, we record a basic but useful closure property of the class of $p$-quasi-Frobenius rings. First we note that the infinite product of $p$-quasi-Frobenius rings need not be $p$-quasi-Frobenius. For example, the ring \(R = \prod_{n=1}^{\infty} \mathbb{Z}/2\mathbb{Z}\) fails the descending chain condition on principal ideals, and thereby cannot be perfect despite being a product of quasi-Frobenius (hence p-quasi-Frobenius) rings. For finite products we have the following.

\begin{proposition}\label{prop:product-pQF}
    The class of $p$-quasi-Frobenius rings is closed under finite direct products. That is, if $R_1,\ldots,R_n$ are $p$-quasi-Frobenius rings, then the direct product $R=R_1\times\cdots\times R_n$ is again $p$-quasi-Frobenius.
\end{proposition}

\begin{proof}
    Since each $R_i$ is two-sided endoartinian, equivalently left and right perfect, their finite direct product $R$ is again left and right perfect by the standard closure of perfect rings under finite products. Moreover, the finite direct product of $p$-injective rings is $p$-injective on both sides by \cite{nicholson1995principally}*{Example 2.1(2)}. Hence $R$ is two-sided endoartinian and two-sided $p$-injective and hence is $p$-quasi-Frobenius.
\end{proof}

\section{$p$-Johns Rings and Faith-Menal Conjecture}

We recall that a right noetherian ring in which every right ideal is an annihilator ideal is called a Johns ring. Equivalently, Johns rings are noetherian dual rings in the sense of Kaplansky \cite{kaplansky1948dual}. As in the case of quasi-Frobenius rings, we now consider a principal analogue of Johns rings by introducing $p$-Johns rings. We first introduce the notion of $p$-dual rings, principal analogue of dual rings. 

\begin{definition}\label{def:p-dual}
    A ring $R$ is called \emph{right $p$-dual} (\emph{left $p$-dual}) if every principal right (resp. left) ideal of $R$ is the right (resp. left) annihilator of an element of $R$. $R$ is called \emph{strongly right $p$-dual} (\emph{strongly left $p$-dual}) if $\mathbb{M}_n(R)$ is right (resp. left) $p$-dual for every $n\ge 1$. $R$ is called (strongly) $p$-dual if it is both right and left (resp. strongly) $p$-dual.
\end{definition}

We first observe that right $p$-dual rings need not be left $p$-dual as is shown by the following example.

\begin{example}\label{ex:right-pdual-not-left-pdual}
Let $S=\Bbbk (x)$ be the field of rational functions over $\Bbbk$ and $\sigma:S\to S$ defined by $\sigma(f(x))=f(x^2)$ be the Frobenius monomorphism. Since $\sigma(S)=\Bbbk (x^2)\subsetneq S$, $\sigma$ is injective but not surjective.

Let $M=S$ as an additive group. We give $M$ the structure of an $(S,S)$-bimodule by defining $m\cdot s=ms$, and $s\cdot m=\sigma(s)m$ for all $s\in S$ and $m\in M$. Consider the split extension $R=S\ltimes M$, whose underlying additive group is $S\oplus M$ and whose multiplication is given by $(s_1,m_1)(s_2,m_2) = (s_1s_2,\sigma(s_1)m_2+m_1s_2)$.  Let $J=0\oplus M$. Then $J^2=0$ and $R/J\cong S$ is a field. Hence $J=J(R)$.

We first show that $R$ is right $p$-dual. Let $(0,m)\in J$ with $m\neq0$. Then $(0,m)(s,n)=(0,ms)$ for every $(s,n)\in R$. Since $S$ is a field, $mS=M$, and therefore $(0,m)R=J$. On the other hand, every element outside $J$ is a unit, so every principal right ideal is one of $0, J, R.$ Moreover, $\rann(1,0)=0$, $\rann(0,m)=J$ for  $m\neq0$, and $\rann(0,0)=R.$ Hence every principal right ideal is the right annihilator of an element, and therefore $R$ is right $p$-dual.

However, \((s,n)(0,m)=(0,\sigma(s)m)\), so \(R(0,m)=0\oplus \sigma(S)m\). Since \(\sigma(S)=\Bbbk (x^2)\subsetneq S\), the left ideal \(K=R(0,m)\) is a proper nonzero principal left ideal contained in \(J\). Finally, \(\lann(1,0)=0\), \(\lann(0,m)=J\) for \(m\neq0\), and \(\lann(0,0)=R\), so the only left annihilators are \(0\), \(J\), and \(R\). Since \(K\notin\{0,J,R\}\), the principal left ideal \(K\) cannot be the left annihilator of any element. Consequently, \(R\) is not left \(p\)-dual.

\end{example}

Next, we show that the concepts of dual rings and $p$-dual rings are logically independent, that is, neither implies the other, in general even in commutative case.

\begin{example}\label{ex:dual-pdual-independent}
\begin{enumerate}
    \item Let $R=\prod_{i=1}^{\infty} \Bbbk$ be the direct product of countably many copies of the field with two elements. Every element of $R$ is an idempotent and for any idempotent $e$, the principal ideal $eR=\ann(1-e)$. Thus, $R$ is $p$-dual. 

    Consider $I = \bigoplus_{i=1}^{\infty} \Bbbk$, the ideal of sequences in $R$ with finite support. First note that $\ann(I) = 0$. Indeed, if $r = (r_i) \neq 0$, choose an index $j$ with $r_j = 1$, and let $e_j \in I$ be the sequence with a $1$ in the $j^\textnormal{th}$ position and $0$ elsewhere. Then $r e_j \neq 0$. Now, if $I$ were a annihilator, there would exist a left ideal $L$ such that $I = \ann(L)$. Then $IL = 0$. In particular, $e_j L = 0$ for every standard basis vector $e_j \in I$, which forces every element of $L$ to have zero in every coordinate; hence $L = 0$. But then $I = \ann(0) = R$, contradicting the fact that $I \neq R$. Therefore $I$ is not a annihilator, so $R$ is not dual.

    \item Let $\Bbbk$ be a field and set $R=\Bbbk[x,y]/(x^2,y^2)$, which is a 4-dimensional local Frobenius algebra with basis $\{1,x,y,xy\}$. As a Frobenius algebra, $R$ is quasi-Frobenius and hence a dual ring. We will show later that $R$ is not $p$-dual (see Remark \ref{rem:converse-pdual-principal-injective}). 
\end{enumerate}
\end{example}

\begin{remark} \label{rem:p-dual_not_dual}
    Hajarnavis and Norton \cite{hajarnavis1985dual} note that dual rings satisfy several strong structural properties. In contrast, $p$-dual rings are much wilder. None of the classical properties of dual rings carry over automatically. For example, the $p$-dual ring $R$, considered in Example \ref{ex:dual-pdual-independent}(1), is not semiperfect since its Jacobson radical is zero, so $R/J\cong R$ is not semisimple artinian. Moreover, $R$ is not uniserial. Indeed, $R$ contains incomparable principal ideals, for example, the ideals generated by orthogonal idempotents with disjoint supports. Also, the free module $R$ itself which is finitely generated contains an infinite direct sum of nonzero ideals (for example, generated by orthogonal idempotents $e_i$ with disjoint supports). Hence $R$ has infinite uniform dimension as an $R$-module. Finally, in any two-sided dual ring, a classical result implicit in the work of Hajarnavis and Norton asserts that $\soc(R_R)=\lann(J(R))=\rann(J(R))=\soc(_R R)$. This fails in $R$ since $J(R)=0$, so $\lann(J(R))=R$ but $\soc(R_R)=\oplus_{i=1}^{\infty} \Bbbk$ which is a proper ideal.
\end{remark}

We further remark that $p$-dual rings are dual of principally projective or Rickart rings (that is, rings in which every principal right ideal is projective) \cite{hattori1960foundation}. Before we formally introduce $p$-Johns rings, we give some basic properties of right (left) $p$-dual rings. Although the results hold for both right as well as left $p$-dual rings, we only state the results for right $p$-dual rings. Our first result deal with arbitrary product of right $p$-dual rings.

\begin{proposition} \label{prop:direct-product-pdual}
    The direct product of right $p$-dual rings is again right $p$-dual.
\end{proposition}

\begin{proof}
    Consider $R=\prod_{i \in I} R_i$, where each $R_i$ is right $p$-dual. Let $a=(a_i) \in R$. For each $i$, by right $p$-duality of $R_i$, there exists $b_i \in R_i$ such that $a_i R_i=\rann_{R_i}(b_i)$. Set $b=(b_i) \in R$. Then $\rann_R (b)=\{(x_i) \in R \ | \ x_i b_i=0  \ \text{for all} \ i\}=\prod_i \rann_{R_i} (b_i)=\prod_i a_i R_i=aR$. Hence $aR=\rann_R(b)$. Thus, $R$ is right $p$-dual.
\end{proof}

\begin{proposition}\label{prop:pdual-principal-injective}
If $R$ is right $p$-dual, then $R$ is left $p$-injective. In particular, every two-sided $p$-dual ring satisfies $\soc(R_R)=\soc({}_RR)$.
\end{proposition}

\begin{proof}
Let $a\in R$. By the right $p$-dual property, there exists an element $t\in R$ such that $aR=\rann(t)=\{x\in R\mid tx=0\}$. Take any $y\in aR$. Then $y=ar$ for some $r\in R$. Since $a\in aR=\rann(t)$, we have $ta=0$. Hence $ty=t(ar)=(ta)r=0$. Thus $t\in\lann(aR)$ and therefore $\rann(\lann(aR))\subseteq\rann(t)=aR$. The reverse inclusion is immediate. Hence $\rann(\lann(aR))=aR$ and consequently $R$ is left $p$-injective.

If $R$ is two-sided $p$-dual, then $R$ is both left and right $p$-injective. Therefore, by \cite{lam1999lectures}*{Exercise 13.27}, $\soc(R_R)=\soc({}_RR)$.
\end{proof}

We note that the converse of Proposition~\ref{prop:pdual-principal-injective} is not true, in general.

\begin{remark}\label{rem:converse-pdual-principal-injective}
    Consider the ring in Example \ref{ex:dual-pdual-independent}(2). Its unique maximal ideal is $J = (x, y)$, with $\dim_\Bbbk(J) = 3$ and basis $\{x, y, xy\}$, and its socle is $\soc(R) = (xy)$, which is $1$-dimensional. Since $R$ is a finite-dimensional commutative algebra with a $1$-dimensional socle, it is a Frobenius algebra, hence quasi-Frobenius (indeed, a 0-dimensional Gorenstein ring). In particular, $R$ is self-injective and therefore both left and right $p$-injective.

    We show that $R$ is not right $p$-dual for if $R$ were right $p$-dual, then for every $a \in R$, there exists $b \in R$ such that $aR = \ann(b).$ Since $R$ is quasi-Frobenius, taking annihilators on both sides yields $\ann(a) = \ann(aR) = \ann(\ann(b)) = bR,$ which is a principal ideal. That is, the annihilator of every element would have to be a principal ideal.

    Consider the element $xy \in R$. Its annihilator is $\ann(xy) = \{ r \in R \mid rxy = 0 \} = (x, y) = J$, because $x(xy) = x^2 y = 0$ and $y(xy) = x y^2 = 0$, while any element not in $J$ is a unit and cannot annihilate $xy$. However, $J$ is not a principal ideal. Indeed, for any nonzero $f \in J$, the principal ideal $(f)$ is contained in $J$ and, since $J^2 = (xy)$, it is spanned by $f$ and $xy$; hence $\dim_\Bbbk(fR) \le 2$. But $\dim_\Bbbk(J) = 3$. Therefore, $J$ cannot be generated by a single element. Consequently, $\ann(xy) = J$ is not a principal ideal, so $(xy)R$ cannot be expressed as the right annihilator of any single element. Hence $R$ is not right $p$-dual (nor left $p$-dual, by commutativity).

    This example also highlights an important distinction: although $R$ is quasi-Frobenius and therefore $p$-quasi-Frobenius (being two-sided endoartinian and two-sided $p$-injective), it is not $p$-dual. Thus, the class of $p$-dual rings is strictly smaller than the class of $p$-quasi-Frobenius rings, even under the strongest possible finiteness and injectivity assumptions.
\end{remark}

As an analogue of Johns rings and strongly Johns rings, we now introduce the notion of $p$-Johns rings and strongly $p$-Johns rings.

\begin{definition}\label{def:p-Johns}
    A ring is called \emph{right $p$-Johns} (\emph{left $p$-Johns}) if it is right (resp. left) endonoetherian and right (resp. left) $p$-dual. Moreover, A right (left) $p$-Johns ring $R$ is called \emph{strongly right $p$-Johns} (\emph{resp. strongly left $p$-Johns}) if $\mathbb{M}_n(R)$ is right (resp. left) $p$-Johns for every $n\ge 1$. A ring is called $p$-Johns (strongly $p$-Johns) if it is both right and left $p$-Johns (resp. strongly $p$-Johns).
\end{definition}

We remark that a right $p$-Johns ring need not be left $p$-Johns. Recall that the ring $R$ constructed in Example ~\ref{ex:right-pdual-not-left-pdual} is right $p$-dual but not left $p$-dual. It can be seen that $R$ is endonoetherian. To see this, we first note that the principal right and left annihilators of $R$ are precisely $0$, $J$, and $R$. Hence both the ascending chains of principal right annihilators and principal left annihilators stabilize. Hence, the ring is both right and left endonoetherian. Hence $R$ is right $p$-Johns but not left $p$-Johns.

In view of the examples in \cite{gouaid2020endo-noetherian} showing that endonoetherianity does not imply noetherianity in general, the following proposition provides a natural setting in which this implication does hold: namely, a principal right ideal right $p$-Johns ring is right noetherian, and hence right Johns. 

\begin{proposition} \label{prop:endonoetherian-noetherian}
    A right \( p \)-Johns principal right ideal ring is right noetherian.
\end{proposition}

\begin{proof}
Let $R$ be a right \( p \)-Johns principal right ideal ring and let $I_1 \subseteq I_2 \subseteq I_3 \subseteq \cdots$ be an ascending chain of right ideals of \( R \). Since \( R \) is principal right ideal ring, each \( I_k \) is principal, say \( I_k = x_kR \).  Because \( R \) is right \( p \)-Johns, every principal right ideal is the principal right annihilator of some element. Thus for each \( k \) there exists \( a_k \in R \) such that $I_k = x_kR = \rann(a_k)$.

Hence the above chain becomes an ascending chain $\rann(a_1) \subseteq \rann(a_2) \subseteq \rann(a_3) \subseteq \cdots$ of principal right annihilators. Since $R$ is Right endonoetherian, this chain of principal right annihilators stabilizes. Hence \( R \) is right noetherian.
\end{proof}

\begin{proposition}\label{thm:pJohns-perfect}
    Every right \(p\)-Johns is right perfect.
\end{proposition}

\begin{proof}
    Let $R$ be a right \( p \)-Johns ring. Then $R$ is right \(p\)-dual. Hence every principal right ideal is a right annihilator of a single element. Since \(R\) is right endonoetherian, it satisfies the ACC on principal right annihilators. Since every principal right ideal is a principal right annihilator, $R$ has ACC on principal right ideals. By Jonah's characterization \cite{lam2001firstcourse}*{23.20(7), Bass Theorem P}, \(R\) is right perfect.
\end{proof}

We are, now, in a position to give and discuss the principal analogue of the Faith-Menal conjecture, namely, \emph{every strongly right Johns ring is quasi-Frobenius} (see \cite{faith1994structure}). This conjecture arose in response to counterexamples showing that right Johns rings need not be quasi-Frobenius (see \citelist{\cite{faith1992counter} \cite{nicholson2003quasi}*{Example~8.16}}). We consider a refinement of this conjecture based on chain conditions on principal annihilators.

\begin{remark}\label{rem:p-QF_not_strongly_p-dual} 
The $p$-quasi-Frobenius rings need not satisfy $p$-dual and hence the strongly $p$-dual condition, in general.

To see this, consider the $p$-quasi-Frobenius ring from Remark~\ref{rem:p-QF_ex_remark}. We show that the $2$-dimensional $\Bbbk$-vector space $x_1R=\Bbbk x_1\oplus \Bbbk \mathfrak{m}$ is not the annihilator of any single element of $R$. Let $b\in R$. If $b$ is a unit, then $\ann_R(b)=0$. Suppose now that $b$ is a non-unit. Since $R$ is local with Jacobson radical $J(R)=\bigoplus_{i=1}^{\infty}\Bbbk x_i\oplus \Bbbk \mathfrak{m}$, we have $b\in J(R)$. Write $b=\sum_{i\in F}\beta_i x_i+\gamma \mathfrak{m}$, where $F$ is finite. Choose $k\notin F$. Then $x_kb=0$, since $x_kx_i=0$ for all $i\neq k$ and $x_k\mathfrak{m}=0$. Hence $x_k\in \ann_R(b)$.
Moreover, $\mathfrak{m}b=0$, because $\mathfrak{m}x_i=0$ and $\mathfrak{m}^2=0$. Thus $\mathfrak{m}\in \ann_R(b)$ as well. Consequently, $\ann_R(b)$ contains infinitely many linearly independent elements $x_k$, and therefore $\ann_R(b)$ is infinite-dimensional over $\Bbbk$.

We have, thus, shown that for every $b\in R$, $\ann_R(b)=0$ or $\ann_R(b)$ is infinite-dimensional over $\Bbbk$. However, the principal ideal $x_1R$ has dimension $2$ over $\Bbbk$. Therefore, $x_1R\neq \ann_R(b)$ for all $b\in R$. Hence $R$ contains a principal ideal which is not the annihilator of any single element. Thus $R$ fails the $p$-dual condition, and consequently $R$ is not strongly $p$-dual.
\end{remark}

It is known that, in classical setting, a right Johns ring need not be right artinian (see \cite{faith1992counter}) and strongly right Johns rings are not known to be quasi-Frobenius (see \cite{faith1994structure}). In contrast, the principal analogue of Johns and quasi-Frobenius rings exhibits a much more rigid behavior, as is shown in the following theorem.

\begin{theorem}\label{thm:pJohns-pQF}
    Let $R$ be a two-sided $p$-Johns ring. Then $R$ is $p$-quasi-Frobenius. If, in addition, $R$ is a principal two-sided ideal ring, then $R$ is quasi-Frobenius.
\end{theorem}

\begin{proof}
    Since $R$ is a two-sided $p$-Johns ring, $R$ is both right and left endonoetherian, as well as right and left $p$-dual. By Proposition~\ref{prop:pdual-principal-injective}, since $R$ is left $p$-dual, it is right $p$-injective. Being both right endonoetherian and right $p$-injective, $R$ is right endoartinian by \cite{gera2026endoartinian}*{Theorem~2.14}. Dualizing, since $R$ is right $p$-dual, Proposition~\ref{prop:pdual-principal-injective} implies that $R$ is left $p$-injective. Being both left endonoetherian and left $p$-injective, $R$ is left endoartinian by the left-sided analogue of \cite{gera2026endoartinian}*{Theorem~2.14}. Thus, $R$ is left and right endoartinian (hence left and right perfect) and left and right $p$-injective. Consequently, $R$ is $p$-quasi-Frobenius. The second assertion follows by invoking Theorem~\ref{thm:PIR-pQF-QF}.
\end{proof}

We remark that two-sidedness assumption in Theorem~\ref{thm:pJohns-pQF} is essential, that is, a one-sided right $p$-Johns ring need not be $p$-quasi-Frobenius, nor even left $p$-injective, even though it is both left and right endonoetherian. To see this, consider the ring $R$ in Example~\ref{ex:right-pdual-not-left-pdual}. We showed in Example \ref{ex:right-pdual-not-left-pdual} that $R$ is right $p$-dual but not left $p$-dual. We claim that $R$ is not right $p$-injective and hence not $p$-quasi-Frobenius. To see this, let $0\neq m\in M$ and set $a=(0,m)\in J$. Then $\rann_R(a)=\{(s,n)\in R \mid (0,m)(s,n)=0\}=\{(s,n)\in R \mid ms=0\}=J$, since $S$ is a field. Moreover, $\lann_R(\rann_R(a))=\lann_R(J)=\{(s,n)\in R \mid (s,n)J=0\}$. For any $(s,n)\in R$ and $(0,m')\in J$, we have $(s,n)(0,m')=(0,\sigma(s)m')$. This vanishes for all $m'\in M$ if and only if $\sigma(s)=0$, which, by injectivity of $\sigma$, forces $s=0$. Hence $\lann_R(J)=J$. However, the principal left ideal generated by $a$ is $Ra=R(0,m)=0\oplus \sigma(S)m$. Since $\sigma(S)=\Bbbk (x^2)\subsetneq S=\Bbbk (x)$, we have $\sigma(S)m \subsetneq M$, so $Ra$ is a proper nonzero subideal of $J$. Therefore, $\lann_R(\rann_R(a))=J\neq Ra$, which violates the standard characterization of right $p$-injective rings. Consequently, $R$ is not $p$-quasi-Frobenius.

We further note that the implication $p$-quasi-Frobenius $\implies$ strongly $p$-dual fails, in general (see Remark~\ref{rem:p-QF_not_strongly_p-dual}). Thus, the classes of $p$-quasi-Frobenius rings and strongly $p$-dual rings are genuinely distinct. 

We remark that Theorem~\ref{thm:pJohns-pQF} may be viewed as a principal analogue of the Faith--Menal conjecture. Within the class of principal two-sided ideal rings, the two-sided $p$-Johns condition already forces the classical quasi-Frobenius property.

\begin{corollary}\label{cor:PRIR-QF-pQF-pJohns}
    Let $R$ be a principal two-sided ideal ring. Then the following are equivalent:
    \begin{enumerate}
        \item $R$ is quasi-Frobenius;
        \item $R$ is $p$-quasi-Frobenius;
        \item $R$ is two-sided $p$-Johns.
    \end{enumerate}
\end{corollary}

\begin{proof}
The equivalence $(1)\Leftrightarrow(2)$ follows from Theorem~\ref{thm:PIR-pQF-QF} and $(3)\Rightarrow(1)$ follows from Theorem~\ref{thm:pJohns-pQF}.

It remains to prove $(1)\Rightarrow(3)$. Assume that $R$ is quasi-Frobenius. Then $R$ is left and right noetherian and therefore, left and right endonoetherian. We show that $R$ is right $p$-dual. Let $aR$ be a principal right ideal of $R$. Since $R$ is quasi-Frobenius, it is a dual ring, so there exists a left ideal $A$ such that $aR=\rann(A)$. Since $R$ is self-injective, $A=\lann(\rann(A))=\lann(aR)$. As $R$ is a principal left ideal ring, there exists $t\in R$ such that $A=Rt$. Hence $aR=\rann(A)=\rann(Rt)=\rann(t)$. Thus $R$ is right $p$-dual. By symmetry, $R$ is left $p$-dual. Therefore $R$ is two-sided $p$-Johns.
\end{proof}

Note that the principal ideal hypothesis in Theorem~\ref{thm:pJohns-pQF} and Corollary~\ref{cor:PRIR-QF-pQF-pJohns} is essential. Without it, even strong annihilator-theoretic conditions do not force either $p$-quasi-Frobenius or quasi-Frobenius behavior. The following example exhibits a ring which is strongly two-sided $p$-dual but fails to be $p$-quasi-Frobenius and hence fails to be quasi-Frobenius. In particular, strong $p$-duality alone is far weaker than the two-sided $p$-Johns condition. This also shows that hypothesis in Theorem \ref{thm:pJohns-pQF} cannot be extended to $p$-dual rings.

\begin{example}\label{ex:strongly-pdual-not-pQF}
Let $R=\prod_{i=1}^{\infty}\Bbbk$. From Example \ref{ex:dual-pdual-independent}(1), $R$ is $p$-dual but not dual. We show that $R$ is strongly $p$-dual. For each $n\geq 1$, $\mathbb M_n(R)\cong \mathbb M_n\left(\prod_{i=1}^{\infty}\Bbbk \right) \cong \prod_{i=1}^{\infty}\mathbb M_n(\Bbbk)$. We first show that $\mathbb M_n(\Bbbk)$ is $p$-dual. 

Let $A=\mathbb M_n(\Bbbk)$ and let $aA$ be a principal ideal of $A$. Since $A$ is semisimple artinian, every ideal of $A$ is generated by an idempotent. Hence there exists an idempotent $e\in A$ such that $aA=eA$. It is known that $eA=\ann_A(1-e)$. Thus every principal ideal of $A$ is the annihilator of a single element, namely $1-e$. Therefore $A$ is $p$-dual. Thus, by Proposition~\ref{prop:direct-product-pdual} (in commutative case), $\mathbb M_n(R) \cong \prod_{i=1}^{\infty}\mathbb M_n(\Bbbk)$ is $p$-dual for every $n\geq 1$. Hence $R$ is strongly $p$-dual.

Next we claim that $R$ is not $p$-quasi-Frobenius. For each $k\geq 1$, define $e_k=(\underbrace{0,\cdots,0}_{k-1},1,1,1,\cdots)$. Then $e_kR= \{(a_i)\in R : a_1=\cdots=a_{k-1}=0\}$. Consequently, $e_1R \supseteq e_2R \supseteq e_3R \supseteq \cdots$. It can be seen that the inclusions are strict. To see this, let $u_k=(0,\cdots,0,1,0,0,\cdots)$, where the only nonzero coordinate occurs in position $k$. Since $u_k=e_k u_k$, we have $u_k\in e_kR$. However, $u_k\notin e_{k+1}R$, because every element of $e_{k+1}R$ has $k$-th coordinate equal to $0$. Hence $e_1R \supsetneq e_2R \supsetneq e_3R \supsetneq \cdots$ is a strictly descending chain of principal right ideals. Therefore $R$ is not right endoartinian. Since $R$ is commutative, it is also not left endoartinian. Thus $R$ is not two-sided endoartinian, and hence is not $p$-quasi-Frobenius.
\end{example}

\begin{remark}
The ring $R=\prod_{i=1}^{\infty}\Bbbk $ in Example \ref{ex:strongly-pdual-not-pQF} is not a principal ideal ring. Indeed, consider the ideal $I=\bigoplus_{i=1}^{\infty}\Bbbk$, consisting of all sequences with finite support. Suppose that $I=aR$ for some $a=(a_1,a_2,\cdots)\in R$. Since $e_i=(0,\cdots,0,1,0,\cdots)\in I=aR$ for every $i$, there exists $r^{(i)}\in R$ such that $e_i=ar^{(i)}$. Examining the $i$-th coordinate gives $1=(ar^{(i)})_i=a_i r^{(i)}_i$. Since $a_i,r^{(i)}_i\in\Bbbk$, we obtain $a_i=1$ for every $i$. Hence $a=(1,1,1,\cdots)=1$, and therefore $aR=R$, contradicting $I\neq R$. Thus the ideal $I$ is not principal.
\end{remark}

\section{A Reduction of Faith’s Conjecture via $p$-quasi-Frobenius Rings}

In this section, we consider the principal analogue of Faith's conjecture. We first give certain conditions under which Kasch rings are $p$-quasi-Frobenius. We remark that our proof mimics the proof of \cite{nicholson2003quasi}*{Lemma 5.21}.

\begin{theorem} \label{thm:semiprimary_PRIR}
    Let $R$ be a semiprimary principal right ideal ring. If $R$ is right injective and right Kasch, then $R$ is $p$-quasi-Frobenius. Moreover, if, in addition, $R$ is also a principal left ideal ring, then $R$ is quasi-Frobenius.
\end{theorem}

\begin{proof}
    We show that $R$ is left $p$-injective. Let $a, b\in R$ be such that $\lann(a)\subseteq \lann(b)$. We prove that $b \in a R$. Suppose, on the contrary, that $b\not\in aR$. Consider the right ideal $I=aR+bR$. Since $R$ is a principal right ideal ring, $I$ is principal. since $b\not\in aR$, the quotient $I/aR$ is nonzero. Choose a maximal submodule $M/aR$ of $I/aR$. Then $M$ is a submodule of $I$ such that $aR \subseteq M \subsetneq I$ and $I/M$ is a simple right $R$-module.

    Since $R$ is right Kasch, every simple right module embeds into $R_R$. Hence there exists a monomorphism $\sigma:I /M \hookrightarrow R_R$. Define a right $R$-module homomorphism $\gamma:I\to R$ by $\gamma(x)=\sigma(x+M)$. Since, $I$ is a principal right ideal and $R_R$ is right injective $\gamma$ extends to a $R_R$ extends to a homomorphism from $R_R$ to itself, that is, left multiplication by some element of $R$, that is, there exists $c \in R$ such that $\gamma(x)=cx$ for all $x \in I$. 

    Since $a\in M \subsetneq I$, we have $ca=\gamma(a)=\sigma(a+M)=\sigma(0)=0$. So, $c\in \lann(a)$. By our assumption $\lann(a)\subseteq \lann(b)$, it follows that $c\in \lann(b)$, hence $cb=0$. However, since $b\not\in M$, the element $b+M$ is nonzero in the simple module $I/M$. As $\sigma$ is injective, $\sigma(b+M)\neq 0$. Thus $cb=\gamma(b)=\sigma(b+M)\neq 0$, which contradicts that $cb=0$. Therefore, our supposition was false, and $b\in aR$. Thus $R$ is left $p$-injective and hence $R$ is $p$-injective on both sides.

    Since $R$ is semiprimary, it is left and right perfect. Therefore $R$ is $p$-quasi-Frobenius. If $R$ is also a principal left ideal ring, then Theorem~\ref{thm:PIR-pQF-QF} yields that $R$ is quasi-Frobenius.
\end{proof}

We remark that the first part of Theorem~\ref{thm:semiprimary_PRIR} provides an affirmative answer, within the class of two-sided principal ideal rings, to Question~3.1 (Q2) of Faith \cite{faith2002when}:

\begin{quote}
\emph{Question 3.1 (Q2). Is a right and left perfect right self-injective ring $R$ necessarily quasi-Frobenius?}
\end{quote}

and the second part establishes the validity of Faith's conjecture within the class of semiprimary principal ideal rings. We record this in the following corollary.

\begin{corollary}\label{cor:pir_kasch_qf}
    Let $R$ be a two-sided principal ideal ring which is right Kasch. Suppose that one of the following conditions holds:
    \begin{enumerate}
        \item $R$ is right and left perfect and right self-injective;
        \item $R$ is semiprimary and right self-injective.
    \end{enumerate}
    Then $R$ is quasi-Frobenius.
\end{corollary}

We construct an example showing that, in Theorem~\ref{thm:semiprimary_PRIR}, the principal right ideal ring hypothesis cannot be omitted and right self-injectivity cannot be weakened to right $p$-injectivity. The construction generalizes the example of algebra of dual numbers in Remark~\ref{rem:chain_symmetry_semiprime} from the commutative case $\sigma = \id$ to a non-surjective endomorphism $\sigma$. It is precisely this non-surjectivity that produces a principal right ideal $xR=\sigma(\mathbb{F})x\subsetneq \mathbb{F} x$ which is not a right annihilator, breaking right $p$-duality while leaving right $p$-injectivity intact.

\begin{example} \label{ex:semiprimary_PRIR}
    Let $\mathbb{F}=\mathbb{K} (t)$ be the field of rational functions and let $\sigma:\mathbb{F}\to \mathbb{F}$ be the injective endomorphism defined by $\sigma(t)=t^2$. Consider the skew polynomial ring $R = \mathbb{F}[x;\sigma]/(x^2)$. Note that every element of $R$ has the form $a+bx$ with $a,b\in \mathbb{F}$, where $x^2=0$ and $xa=\sigma(a)x$ for all $a\in \mathbb{F}$.

    \begin{enumerate}
        \item \textit{Semiprimary.} Let $J$ be the Jacobson radical. Note that $J=\mathbb{F}x$. We have $J^2=0$ and $R/J\cong \mathbb{F}$, which is a field. Hence $R/J(R)$ is semisimple artinian and $J(R)$ nilpotent Jacobson radical, so $R$ is semiprimary.
    
        \item \textit{Right $p$-injective.} If $\alpha=a+bx$ with $a\neq0$, then $\alpha$ is a unit, so $\rann(\alpha)=0$ and $\lann(\rann(\alpha))=R=R\alpha$. If $\alpha=bx$ with $b\neq0$, then $\rann(\alpha) = \rann(bx)=\mathbb{F}x=J$. Consequently, $\lann(\rann(bx))=\lann(J)=J$. On the other hand $R(\alpha) = R(bx)=(\mathbb{F}+\mathbb{F}x)bx=\mathbb{F}x=J$. Thus $\lann(\rann(\alpha))=R\alpha$ for every $\alpha\in R$, and $R$ is right $p$-injective.
    
        \item \textit{Principal left ideal ring.} The left ideals of $R$ are $0$, $\mathbb{F}x$, and $R$, and each of these is generated by a single element. Hence $R$ is a principal left ideal ring.
    
        \item \textit{Not left $p$-injective.} Consider $\alpha=x$. Then $\lann(x)=\mathbb{F}x=J$. Hence $\rann(\lann(x))=\rann (J)=\mathbb{F}x$. However $xR=x\mathbb{F}=\{\sigma(v)x\mid v\in \mathbb{F}\}=\sigma(\mathbb{F})x$. Since $\sigma(\mathbb{F})=\Bbbk(t^2)\subsetneq \Bbbk(t)=\mathbb{F}$, we obtain $xR=\sigma(\mathbb{F})x\subsetneq \mathbb{F}x$. Hence $xR\neq \rann(\lann(x))$, so $R$ is not left $p$-injective. Hence $R$ is not $p$-quasi-Frobenius.

        \item \textit{Not a principal right ideal ring.} The Jacobson radical is $J=\mathbb{F}x=\{xv\mid v\in \mathbb{F}\}$. For any element $a=bx\in J$ we have $aR=(bx)R=\{bx(u+vx)\mid u,v\in \mathbb{F}\}=b\sigma(\mathbb{F})x$. Since $\sigma(\mathbb{F})\subsetneq \mathbb{F}$, the set $b\sigma(\mathbb{F})x$ is a proper subset of $\mathbb{F}x=J$. Thus no single element $a\in J$ satisfies $aR=J$. Hence $J$ is not a principal right ideal and $R$ is not a principal right ideal ring.

        \item \textit{Right Kasch.} The unique maximal right ideal of \(R\) is \(J=\mathbb F x\), since \(R/J\cong\mathbb F\) is a field. For \(a+bx\in R\) and \(c\in\mathbb F\), $(a+bx)(cx)=acx+bx\cdot cx=acx+b\sigma(c)x^2=acx.$ Hence \(a+bx\in\lann(J)\) if and only if \(ac=0\) for all \(c\in\mathbb F\), which forces \(a=0\). Thus $\lann(J)=J\neq 0,$ so \(R\) is right Kasch by \cite{nicholson2003quasi}*{Proposition 1.44}.

        \item \textit{Not quasi-Frobenius.} Let $K=xR=\sigma(\mathbb{F})x$. Then $\lann(K)=\mathbb{F}x$ and $\rann(\lann(K))=\rann(\mathbb{F}x)=x\mathbb{F}$. Since $\sigma(\mathbb{F})\subsetneq \mathbb{F}$, we have $K=\sigma(\mathbb{F})x\subsetneq x\mathbb{F}=\rann(\lann(K))$. Thus $\rann(\lann(K))$ $\neq K$. By the Ikeda-Nakayama characterization, $R$ is not quasi-Frobenius.
    \end{enumerate}
\end{example}

\begin{remark} \label{rem:prop22-example}
    Example~\ref{ex:semiprimary_PRIR} is a precise one-sided counterpart to Example~\ref{ex:right-pdual-not-left-pdual}.  Indeed, let \(R=\mathbb{F} [x;\sigma]/(x^2)\) with \(\mathbb{F}=\mathbb{K} (t)\) and \(\sigma(t)=t^2\), and let \(J=\mathbb{F} x\). As shown in Example~\ref{ex:semiprimary_PRIR}(3), the left ideals of \(R\) are \(0\), \(J\), and \(R\), and each is principal. Their left annihilators are respectively $\lann(1)=0$, $\lann(x)=J$ and $\lann(0)=R$. Hence the set of principal left ideals coincides exactly with the set of left annihilators of elements, so \(R\) is left \(p\)-dual.

    By Example~\ref{ex:semiprimary_PRIR}(5), the Jacobson radical \(J\) is not a principal right ideal. More decisively, the principal right ideal \(xR=\sigma(\mathbb{F})x=\mathbb{K}(t^2)x\) is a proper nonzero subideal of \(J\), since \(\sigma(\mathbb{F})\subsetneq \mathbb{F}\). However, as computed in the example, the only right annihilators in \(R\) are $\rann(\alpha)\in \{0,\ J,\ R\}$ for every \(\alpha\in R\). Since \(xR\notin\{0,J,R\}\), the principal right ideal \(xR\) cannot be expressed as the right annihilator of any element. Thus \(R\) is not right \(p\)-dual.

    This complements Example~\ref{ex:right-pdual-not-left-pdual}, where the split extension \(S\ltimes M\) was shown to be right \(p\)-dual but not left \(p\)-dual. Together, these two examples demonstrate that the \(p\)-dual property is genuinely one-sided: neither left \(p\)-duality nor right \(p\)-duality implies the other. Furthermore, the present example shows that even a semiprimary, right $p$-injective ring (see Example~\ref{ex:semiprimary_PRIR}(2)) need not be right \(p\)-dual, highlighting the sharp gap between $p$-injectivity and the full \(p\)-dual condition.

    This example also demonstrates the phenomenon described in Proposition~\ref{prop:annihilator-endomorphisms-revised}(3). Taking $g=\id_{R_R}$ and $T=xR$, the proposition yields the canonical inclusion $xR\subseteq\rann(\lann(xR))$. Indeed, $\lann(xR)=\lann(\sigma(\mathbb{F})x)=J$, while $\rann(J)=J$. Hence $xR=\sigma(\mathbb{F})x \subsetneq J= \rann(\lann(xR))=J$, so the inclusion is strict. Equivalently, the pointwise double annihilator condition $\rann(\lann(x))=xR$ fails. The obstruction is precisely the non-surjectivity of $\sigma$, which forces $\sigma(\mathbb{F})\subsetneq \mathbb{F}$, thereby producing the proper containment $xR=\sigma(\mathbb{F})x\subsetneq \mathbb{F}x=J$.
\end{remark}

\section{Perfectness of CF Rings}

Recall that a ring $R$ is \emph{right CF} (resp. \emph{right FGF}) if every cyclic (resp. finitely generated) right $R$-module embeds in a free module. In this section, our goal is to study principal analogue of the FGF conjecture (CF conjecture), namely, every right FGF ring (resp. CF ring) is $p$-quasi-Frobenius (resp. perfect). Before we discuss the principal analogue, we discuss the progress made in the case of CF conjecture and also give some results that help in our explorations. Note that if the CF conjecture holds, then, by the Faith–Walker Theorem \cite{faith1967direct-sum}, the rings are quasi-Frobenius, thus resolving the FGF conjecture. Also in the case CS (ACS) rings, it is known that right CS (resp. ACS) right CF rings are right artinian (see \cite{pardo1998torsionless}, \cite{li2013when}). One may ask what happens if we replace CS (ACS) condition with \emph{pure extending} condition, that is, every pure submodule is essential in the direct summand. Note that every CS module is pure extending but the converse is not true. ACS modules and pure extending modules, on the other hand, are completely independent, as is shown by the following example. 

\begin{example} \label{ex:ACS-PCS}
    \begin{enumerate}
        \item \emph{ACS ring but not pure extending}. Let $\mathbb{F}$ be a field with a proper subfield $\mathbb{K}$. Define the commutative von Neumann regular ring $R = \left\{ (a_n) \in \prod_{n=1}^\infty \mathbb{F} \;\middle|\; a_n \text{ is eventually in } \mathbb{K} \right\}.$ Let $S = \mathbb{M}_2(R)$ be the $2 \times 2$ matrix ring over $R$. Then $S$ is left ACS but not extending by \cite{li2013when}*{Example 2.17}. Since $S$ is von Neumann regular it is not pure extending by \cite{gera2025modules-2}*{Proposition 2.8}.

        \item \emph{Pure extending ring but not ACS}. Let $R=\begin{pmatrix}
\mathbb{Z} & \mathbb{Z}\\
0 & \mathbb{Z}
\end{pmatrix}$. The ring $R$ is right pure-extending (see \cite{gera2025modules-2}*{Example 2.2(2)}). We prove that it is not right ACS. First we note that the only non-trivial idempotent in $R$ are $e_n=\begin{pmatrix}
1 & n\\
0 & 0
\end{pmatrix};$
and $f_n=\begin{pmatrix}
0 & n\\
0 & 1
\end{pmatrix}$, where $n$ is any integer. Thus, every proper direct summand of $R_R$ is of the form $e_nR$ or $f_nR$ for some $n\in \mathbb{Z}$.

Now let $y=\begin{pmatrix}
2 & 1\\
0 & 0
\end{pmatrix}$. A direct computation gives $\rann_R(y)=
\left\{
\begin{pmatrix}
0 & v\\
0 & -2v
\end{pmatrix}
\;\middle|\; v\in \mathbb{Z}
\right\}$. We show that $\rann_R(y)$ intersects every proper direct summand trivially. Note that for any integet $n$,
$$e_nR=
\left\{
\begin{pmatrix}
a & an+b\\
0 & 0
\end{pmatrix}
\;\middle|\; a,b\in \mathbb{Z}
\right\}.$$ 
If an element of $e_nR$ lies in $\rann_R(y)$, then $a=0$, forcing $v=0$, so the element is zero. The case of $f_nR$ is similar. Thus $\rann_R(y)$ intersects every proper direct summand trivially; it is, therefore, not essential in any proper direct summand. Also $\rann_R(y)\neq R$, so it is not essential in $R$ itself. Hence $R$ is not right ACS. 
    \end{enumerate}
\end{example}

To discuss the CF conjecture in the case of pure extending modules, we need to overcome the gap between flatness and projectivity that arises when working with pure submodules. One possible way to do this is to have additional requirements of coherent and AIP conditions. The following theorem captures the synergy of coherence and AIP conditions providing the precise mechanism needed to force pure right ideals and hence the relevant annihilators to be direct summands

\begin{theorem} \label{thm:CF-PCS}
    Let \(R\) be a right AIP, right CF and right coherent ring. If \(R_R\) is pure-extending, then \(R\) is right perfect.
\end{theorem}

Before we prove this theorem, we show how the coherence condition bridges the gap between flatness and projectivity to force pure right ideals to be direct summands. We include this as a lemma which is also of independent interest.

\begin{lemma} \label{lm:PDS} 
    Under the assumptions of Theorem \ref{thm:CF-PCS}, every pure right ideal of \(R\) is a direct summand. Hence \(R\) is a right PDS ring.
\end{lemma}

\begin{proof}
    Let \(I \subseteq R_R\) be a pure right ideal. By the pure-extending hypothesis, there exists an idempotent \(e \in R\) such that $I$ is essential in $eR$. Consider the cyclic right module \(eR/I\). Since \(R\) is right CF, so \(eR/I\) embeds into a free module \(R^n\). Let $f: eR/I \hookrightarrow R^n$ be a monomorphism and set \(v = f(e+I) \in R^n\). Then \(eR/I \cong vR\). Since \(e+I = (e+I)e\), we have \(v = ve\) and hence $v(1-e) = 0$. Because \(I\) is pure in \(R\), the quotient \(R/I\) is flat. Since $R/I \cong eR/I \oplus (1-e)R$, the direct summand \(eR/I\) is flat. Let \(M := eR/I\).

    Now, \(M\) is cyclic and hence finitely generated. Since \(M\) embeds into \(R^n\), and \(R^n\) is finitely presented, the coherence of \(R_R\) implies that \(M\) , being finitely generated submodule of a finitely presented module, is finitely presented. Thus, \(M\) is both flat and finitely presented. By \cite{lam1999lectures}*{Theorem 4.30}, \(M\) is projective. Therefore the short exact sequence $0 \to I \to eR \to M \to 0$ splits and hence \(I\) is a direct summand of \(eR\). Since $I$ is essential in $eR$, we get \(I = eR\). Hence, every pure right ideal is a direct summand, so \(R\) is right PDS.
\end{proof}

Another consequence of the assumptions in Theorem \ref{thm:CF-PCS} is the absence of infinite set of orthogonal idepotents, as is shown in out next proposition. 
\begin{proposition} \label{prop:noIdempotents}
    Under the assumptions of Theorem \ref{thm:CF-PCS}, \(R\) contains no infinite set of nonzero orthogonal idempotents.
\end{proposition}

\begin{proof}
    Assume, for contradiction, that \(\{e_1, e_2, \dots\}\) is an infinite set of nonzero orthogonal idempotents. Define $I = \bigoplus_{i=1}^\infty e_i R$. For each \(n\), let \(F_n = \bigoplus_{i=1}^n e_i R\). Each \(F_n\) is a direct summand of \(R\) (with idempotent \(e_1 + \cdots + e_n\)), hence each \(F_n\) is pure in \(R\). Since \(I = \bigcup_{n=1}^\infty F_n\) is a direct union of pure submodules, \(I\) is pure. By Lemma~\ref{lm:PDS}, \(I\) is a direct summand. Thus, \(I = fR\) for some idempotent \(f \in I\). Since \(f \in      \bigoplus_{i=1}^\infty e_i R\), there exists \(m\) such that $f \in \bigoplus_{i=1}^m e_i R$. For any \(j > m\), \(e_j \in I = fR\), so \(e_j = f r\) for some \(r \in R\). Multiplying on the left by \(e_j\), we get $e_j = e_j^2 = e_j f r$. But \(f \in \bigoplus_{i=1}^m e_i R\) and \(e_j\) is orthogonal to each \(e_i\) for \(i \le m\), so \(e_j f = 0\). Hence \(e_j = 0\), a contradiction.
\end{proof}

We are now in a position to give the proof of Theorem \ref{thm:CF-PCS}.
\begin{proof}[Proof of Theorem \ref{thm:CF-PCS}]
    For any $a \in R$, consider the two-sided ideal $I = RaR$. By right AIP hypothesis, $\rann_R(I)$ is a pure right ideal of $R$. For any $r \in R$, $r \in \rann_R(RaR)$ if and only if $(xay)r = 0$ for all $x, y \in R$ if and only if $ayr = 0$ for all $y \in R$ if and only if $r \in \rann_R(aR)$. Thus $\rann_R(aR) = \rann_R(RaR)$ is a pure right ideal of $R$.
    
    By Lemma~\ref{lm:PDS}, every pure right ideal of $R$ is a direct summand of $R_R$. Hence, for each $n \ge 1$, there exists an idempotent $e_n \in R$ such that $\rann_R(a_n R) = e_n R.$ Applying $\rann_R(-)$ to the ascending chain $a_1 R \subseteq a_2 R \subseteq a_3 R \subseteq \dots$ yields a descending chain of direct summands $e_1 R \supseteq e_2 R \supseteq e_3 R \supseteq \cdots$. By Proposition~\ref{prop:noIdempotents}, $R$ contains no infinite set of nonzero orthogonal idempotents, so $R_R$ satisfies DCC on direct summands. Consequently, the chain of right annihilators stabilizes at some index $k$: $\rann_R(a_k R) = \rann_R(a_{k+1} R) = \cdots = e_k R.$

    Set $f = 1 - e_k$ and consider the corner ring $S = f R f$, which has identity element $f$. Since $a_k R \subseteq a_{k+1} R$, there exists $r \in R$ such that $a_k = a_{k+1} r$. Because $a_k e_k = 0$ and $a_{k+1} e_k = 0$, we have $a_k f = a_k$ and $a_{k+1} f = a_{k+1}$. Multiplying on the right by $f$ yields: $a_k = a_k f = a_{k+1} r f = (a_{k+1} f) r f = a_{k+1} (f r f).$ Set $x = f r f \in S$. Then $a_k = a_{k+1} x$. Now, let $y \in S = f R f$. If $x y = 0$, then: $a_k y = (a_{k+1} x) y = a_{k+1} (x y) = a_{k+1} \cdot 0 = 0.$ This implies $y \in \rann_R(a_k R) = e_k R$. Therefore, $y \in f R f \cap e_k R = \{0\}$, so $y = 0$. This proves that $x$ is a left non-zero-divisor in $S$. Since $R$ contains no infinite set of nonzero orthogonal idempotents, the corner ring $S = f R f$ likewise contains no infinite set of nonzero orthogonal idempotents. Consequently, $S$ is Dedekind-finite, which implies that every left non-zero-divisor in $S$ is a unit in $S$. Thus, $x \in U(S)$ and hence there exists $z \in S$ such that $x z = f$. Multiplying $a_k = a_{k+1} x$ on the right by $z$ gives $a_{k+1} = a_{k+1} f = a_{k+1} (x z) = (a_{k+1} x) z = a_k z \in a_k R.$ Hence, $a_{k+1} R \subseteq a_k R$, which forces $a_k R = a_{k+1} R$,  so the ascending chain of principal right ideals stabilizes at $k$. By Jonah's characterization \cite{lam2001firstcourse}*{23.20(7), Bass Theorem P}, $R$ is right perfect.
\end{proof}

Recall that every semi-hereditary ring is coherent, but the converse is false, in general. Strengthening the coherence hypothesis to semi-hereditary in Theorem~\ref{thm:CF-PCS} not only removes the need for the pure-extending and AIP assumptions but also yields a much stronger structural conclusion, as shown below.

\begin{corollary}\label{thm:CF-semihereditary}
    Let \(R\) be a right CF and right semi-hereditary ring. Then \(R\) is semisimple artinian.
\end{corollary}

\begin{proof}
    Let \(M\) be any cyclic right \(R\)-module. By right CF, \(M\) embeds into a free module. Since \(M\) is finitely generated, its image lies in a finitely generated free module \(R^n\). As \(R\) is right semi-hereditary, $M$, being finitely generated submodule of a projective module, is projective. Hence, every cyclic right \(R\)-module is projective. Now take any right ideal \(I \leq R_R\). The quotient \(R/I\) is cyclic and hence projective. Therefore, the short exact sequence $0 \to I \to R \to R/I \to 0$ splits, and hence \(I\) is a direct summand of \(R_R\). Thus every right ideal of $R$ is a direct summand. Consequently, \(R\) is semisimple artinian.
\end{proof}

We, now, turn our attention to the principal analogue of the CF conjecture as it relates to $p$-quasi-Frobenius rings and $p$-Johns rings. We note that being $p$-quasi-Frobenius does not force the ring to be CF (see Remark \ref{rem:non-CF}).

\begin{remark}\label{rem:non-CF}
    Consider the cyclic module $M=R/(x_1,x_2)$ as considered in Example \ref{ex:explicit-nonreflexive-module}. Its socle is $\soc(M)=\{\bar u\in M \mid J\bar u=0\}.$ For each $i\ge 3$, the element $\bar x_i$ lies in $\soc(M)$, because $x_j x_i=0$ for $j\neq i$ and $x_i^2=\mathfrak m\in (x_1,x_2)$. The set $\{\bar x_i\}_{i\ge 3}$ is linearly independent over $\Bbbk$, so $\dim_{\Bbbk}\soc(M)=\infty$. On the other hand, for any finite $n$, $\soc(R^n)\cong (\soc(R))^n \cong \mathbb \Bbbk^n,$ since $\soc(R)=\ann_R(J)=\mathbb \Bbbk\mathfrak m$ is $1$-dimensional. If $M$ embedded into a free module, its image would lie in $R^n$ for some finite $n$ (because the image of the cyclic generator has finite support), forcing an embedding of infinite-dimensional $\soc(M)$ into finite-dimensional $\soc(R^n)$, which is not possible. Hence $R$ is not CF.
\end{remark}

In the case of right \(p\)-dual rings (more generally, right $p$-Johns rings), we note that, by Proposition~\ref{prop:pdual-principal-injective}, every right \(p\)-dual ring is left $p$-injective. On the other hand, the left-handed version of \cite{nicholson2003quasi}*{Lemma 7.2(1)} asserts that every right CF ring is left $p$-injective. Thus, both classes sit inside the family of left $p$-injective rings. Despite this common behavior, the two notions are logically independent, as shown in the following note.

\begin{note} \label{note:independence_pdual_CF}
    \begin{enumerate}
        \item The ring \(R=\prod_{i=1}^{\infty} \Bbbk\) from Example~\ref{ex:dual-pdual-independent}(1) is \(p\)-dual. However, it is not CF, since the ideal of finitely supported sequences \(I=\bigoplus_{i=1}^{\infty}\Bbbk\) cannot be expressed as the annihilator of a finite subset (its annihilator is \(0\), and any annihilator of a finite set in this Boolean ring is a principal ideal generated by an idempotent with finite support, which is strictly contained in \(I\)).
        
        \item The finite-dimensional local algebra \(R=\Bbbk[x,y]/(x^2,y^2)\) from Example~\ref{ex:dual-pdual-independent}(2) is quasi-Frobenius and hence CF. However, as shown in Remark~\ref{rem:converse-pdual-principal-injective}, it is not \(p\)-dual.
    \end{enumerate}
\end{note}

In our next theorem, we give conditions under which the classes of CF rings and $p$-quasi-Frobenius rings coincide.

\begin{theorem} \label{cor:principalCF-QF}
    Let $R$ be a principal two-sided ideal ring. If $R$ is left and right AIP, left and right coherent, and both $R_R$ and ${}_RR$ are pure-extending, then $R$ is two-sided CF if and only if $R$ is $p$-quasi-Frobenius. In particular, $R$ is quasi-Frobenius.
\end{theorem}

\begin{proof}
    First suppose $R$ is two-sided CF. By Theorem~\ref{thm:CF-PCS}, $R$ is right perfect; the left-handed version gives left perfectness. Since right CF rings are left $p$-injective, and left CF rings are right $p$-injective by \cite{nicholson2003quasi}*{Lemma 7.2(1)}, $R$ is two-sided $p$-injective. Hence $R$ is $p$-quasi-Frobenius.
    
    The other direction follows by Theorem \ref{thm:PIR-pQF-QF} and the fact that quasi-Frobenius rings are two-sided CF.
\end{proof}

Theorem \ref{thm:pJohns-pQF} established that two-sided $p$-Johns rings are $p$-quasi-Frobenius. The next result shows that a right $p$-Johns ring is $p$-quasi-Frobenius provided it is also right CF and left perfect. If the hypothesis were strengthened from right CF to left CF, then the Faith–Walker theorem \cite{faith1967direct-sum} would force $R$ to be quasi-Frobenius. However, our setting requires only right CF and left perfect, which allows us to bypass the stronger quasi-Frobenius conclusion and, instead, obtain the intermediate $p$-quasi-Frobenius property.

\begin{theorem} \label{thm:p-QF}
    If $R$ is right $p$-Johns, left CF and left perfect, then $R$ is $p$-quasi-Frobenius.
\end{theorem}

\begin{proof}
    Since \(R\) is right \(p\)-Johns, it is right \(p\)-dual, so by Proposition~\ref{prop:pdual-principal-injective}, \(R\) is left $p$-injective; and by Proposition~\ref{thm:pJohns-perfect}, \(R\) is right perfect. Since \(R\) is left CF, by \cite{nicholson2003quasi}*{Lemma 7.2(1)}, \(R\) is right $p$-injective. The hypothesis also gives that \(R\) is left perfect. Now, since \(R\) is left $p$-injective and right perfect (equivalently, right endoartinian), by the symmetric version of \cite{gera2026endoartinian}*{Theorem 2.14}, \(R\) is left endonoetherian. Since \(R\) is right \(p\)-Johns, it is right endonoetherian, by definition. Therefore, $R$ is \(p\)-quasi-Frobenius.
\end{proof}

Finally, we remark that the left CF property alone, even when combined with semiprimary (hence perfect) assumption, is not sufficient to force the \(p\)-quasi-Frobenius conclusion, as is shown in the following note.

\begin{note} \label{note:leftCF_not_p-QF}
    In Remark~\ref{rem:prop22-example}, we showed that the ring in Example~\ref{ex:semiprimary_PRIR} is left \(p\)-dual but not right \(p\)-dual. Using the left-handed version of \cite{nicholson2003quasi}*{Lemma 7.2(1)}, a ring is left CF if and only if every left ideal is the left annihilator of a finite subset. For \(R=\mathbb{F}[x;\sigma]/(x^2)\), the left ideals are \(0\), \(J=\mathbb{F}x\), and \(R\). Since $0=\lann(1)$, $J=\lann(x)$, $R=\lann(0)$, and each of \(\{1\}, \{x\}, \{0\}\) is finite, \(R\) is left CF. Consequently, \(R\) is right $p$-injective (which we already verified in Example~\ref{ex:semiprimary_PRIR}(2)) and right Kasch. However, \(R\) is not \(p\)-quasi-Frobenius, as it fails to be left $p$-injective (Example~\ref{ex:semiprimary_PRIR}(4)). 
    
    Let \(b \in \mathbb{F} \setminus \sigma(\mathbb{F})\). Define a left \(R\)-module map \(f: J \to R\) by $f(mx) = m b x$ for $m \in \mathbb{F}$. For any \(a+cx \in R\) and \(m x \in J\), $f((a+cx)(mx)) = f(amx + c x m x) = f(amx) = a m b x,$ while $(a+cx)f(mx) = (a+cx)(m b x) = a m b x,$ so \(f\) is indeed a left \(R\)-module homomorphism. If \(R\) were left self-injective, \(f\) would extend to a map \(F: R \to R\). Then $F(x) = F(x \cdot 1) = x F(1).$ Write \(F(1) = a + c x\). Then \(x F(1) = \sigma(a)x\). Hence $b x = f(x) = F(x) = \sigma(a)x,$ forcing \(b = \sigma(a) \in \sigma(\mathbb{F})\), a contradiction. Thus \(R\) is not left self-injective. Therefore, it is certainly not quasi-Frobenius. This aligns perfectly with Bj\"{o}rk's classical example \cite{bjork1969rings}, where a similar construction is used to delineate chain conditions and mininjectivity, illustrating the delicate boundary between one-sided and two-sided finiteness and injectivity properties. Thus, the left CF property alone, even when combined with semiprimary (hence perfect) assumptions, is insufficient to force the \(p\)-quasi-Frobenius conclusion. Since \(R\) also fails the right \(p\)-Johns hypothesis (because it is not right \(p\)-dual), this perfectly illustrates why the right \(p\)-Johns hypothesis in Theorem~\ref{thm:p-QF} is essential and cannot be replaced by weaker injectivity or finiteness conditions.
\end{note}

\section{Morita Invariance of Endonoetherianity}

It is well known that many ring-theoretic properties, including ACC on right annihilators, are not preserved under Morita equivalence \cite{lam1999lectures}*{p.~482}. This naturally prompts the question of whether the endonoetherian property, that is, ACC on principal right annihilators is Morita invariant. The following example demonstrates that, in general, it is not.

\begin{example}\label{ex:shepherdson}
Let $\mathbb{K}$ be a field and let $R$ be the $\mathbb{K}$-algebra generated by $s,t,u,v,w,x,y,z$ subject to the relations determined by $AB=I_2$, where $A=
\begin{pmatrix}
s&t\\
u&v
\end{pmatrix}$, and $B=
\begin{pmatrix}
w&x\\
y&z
\end{pmatrix}$.

By a classical construction of Shepherdson \cite{shepherdson1951inverses} (see also \cite{lam1999lectures}*{Exercise 18, p. 19}), $R$ is a domain. Also $BA\neq I_2$. Consequently, $A$ possesses a right inverse but no left inverse in $\mathbb{M}_2(R)$.

Since $R$ is a domain, the right annihilator of every nonzero element of $R$ is zero. Hence the only right annihilators of elements of $R$ are $0$ and $R$, and therefore $R$ is right endonoetherian.

Set $E=I_2-BA$. Since $BA\neq I_2$, we have $E\neq0$, and $AE=A-A(BA)=A-(AB)A=0$. Choose a nonzero column vector $v\in R^2$ occurring in $E$. Then $Av=0$. Let $Y=
\begin{pmatrix}
v & 0
\end{pmatrix}
\in \mathbb{M}_2(R)$. Then $Y\neq0$ and $AY=0$.

For each $n\ge1$, define $X_n=B^nY$. Since $AB=I_2$, it follows inductively that $A^nB^n=I_2$ for all $n\ge1$. Therefore $A^nX_n=A^nB^nY=Y\neq 0$, while $A^{n+1}X_n =A(A^nB^n)Y =AY =0$. 

Hence $X_n\in \rann_{\mathbb{M}_2(R)}(A^{n+1}) \setminus \rann_{\mathbb{M}_2(R)}(A^n)$. Thus, 

$$\rann_{\mathbb{M}_2(R)}(A) \subsetneq \rann_{\mathbb{M}_2(R)}(A^2) \subsetneq \rann_{\mathbb{M}_2(R)}(A^3) \subsetneq \cdots$$
is an infinite strictly ascending chain of right annihilators of single elements.

Thus $\mathbb{M}_2(R)$ is not right endonoetherian. Since $\mathbb{M}_2(R)$ is Morita equivalent to $R$, right endonoetherianity is not Morita invariant.
\end{example}

We now adopt a categorical perspective to examine the behavior of the endonoetherian property under Morita equivalence. While this property is not preserved in general, we show, in our next result, that it is retained under additional hypotheses ensuring that principal right annihilators correspond between the rings. 

\begin{theorem}\label{thm:morita-princ-ann-acc}
Let $F:\mathrm{Mod}\text{-}R\to\mathrm{Mod}\text{-}S$ be a Morita equivalence with quasi-inverse $G$. For each $a\in R$ write $\ell_a:R\to R$ for left multiplication by $a$ and denote $\mathcal{A}_R:=\{\rann_R(a)=\ker(\ell_a)\mid a\in R\}$ the family of principal right annihilators of $R$. Set $P:=F(R)$ and let $\Phi:\morp_R(R)\xrightarrow{\ \cong\ }\morp_S(P)$ be the ring isomorphism induced by $F$ (so $\Phi(\ell_a)=F(\ell_a)$ for each $a\in R$).

Then $F$ carries the family $\mathcal{A}_R$ bijectively onto the family $\mathcal{K}_P:=\{\ker\big(\Phi(\ell_a)\big)\mid a\in R\}$ of kernels of the corresponding endomorphisms of the right $S$-module $P$. Consequently, $R$ satisfies ACC on the family $\mathcal{A}_R$ if and only if the family $\mathcal{K}_P$ satisfies ACC.

Moreover, if $P\cong S$ as right $S$-modules (equivalently, $\morp_S(P)\cong S$ via left multiplication), then the transported kernels in $\mathcal{K}_P$ are precisely the principal right annihilators of $S$, and therefore the property endonoetherian is Morita invariant.
\end{theorem}

\begin{proof}
Since $F$ is a Morita equivalence, it is additive, exact, fully faithful and essentially surjective. Also it admits a quasi-inverse $G$. In particular, for every $R$-module $M$ the natural map $F_{M,M}:\morp_R(M)\xrightarrow{\ \cong\ }\morp_S\big(F(M)\big)$ is a ring isomorphism. Takin $M=R$, we obtain the stated isomorphism $\Phi$.

Fix $a\in R$ and consider the short exact sequence of right $R$-modules $0\longrightarrow\ker(\ell_a)\xrightarrow{\ \iota\ }R\xrightarrow{\ \ell_a\ }R.$ Applying the exact functor $F$ yields the exact sequence $0\longrightarrow F\big(\ker(\ell_a)\big)\xrightarrow{\ F(\iota)\ }P\xrightarrow{\ F(\ell_a)\ }P$, and hence $F\big(\ker(\ell_a)\big)=\ker\big(F(\ell_a)\big)=\ker\big(\Phi(\ell_a)\big).$ Thus $F$ sends the principal right annihilator $\rann_R(a)=\ker(\ell_a)$ to the submodule $\ker\big(\Phi(\ell_a)\big)\le P$.

Because $F$ has a quasi-inverse $G$, the assignment $\rann_R(a)\longmapsto\ker\big(\Phi(\ell_a)\big)$ is bijective when viewed as a map from $\mathcal{A}_R$ onto $\mathcal{K}_P$: indeed, given $K'=\ker(\Phi(\ell_a))\le P$ we recover $G(K')\cong G\big(F(\ker(\ell_a))\big)\cong\ker(\ell_a)$, so $G$ provides the inverse correspondence on submodules.

The bijection preserves inclusions (equivalences of categories preserve subobjects and inclusions), so strictly ascending chains in $\mathcal{A}_R$ correspond exactly to strictly ascending chains in $\mathcal{K}_P$. This proves equivalence of ACC for the two families.

Finally, if $P\cong S_S$, then $\morp_S(P)\cong\morp_S(S_S)\cong S$ via left multiplication. Hence each $\Phi(\ell_a)$ is left multiplication by some $b\in S$, and $\ker\big(\Phi(\ell_a)\big)=\ker(\ell_b)=\rann_S(b).$ Thus $\mathcal{K}_P$ is exactly the family of principal right annihilators of $S$ and the endonoetherian property is preserved by the Morita equivalence.
\end{proof}

\begin{corollary} \label{cor:endonoetherian-morita}
Let \(R\) and \(S\) be Morita equivalent rings. Assume:

\begin{enumerate}
    \item \(R\) is a principal right ideal ring and right \(p\)-Johns;
    \item under the equivalence, the module \(R_R\) corresponds to a module isomorphic to \(S_S\).
\end{enumerate}

Then \(S\) is right endonoetherian. If, in addition, \(S\) is also principal right ideal ring and right \(p\)-Johns, then \(R\) and \(S\) are right noetherian.
\end{corollary}

\begin{proof}
    By Theorem~\ref{thm:morita-princ-ann-acc}, assumption (2) implies that the Morita equivalence sends principal right annihilators of \(R\) bijectively onto those of \(S\). Since \(R\) is right \(p\)-Johns, its principal right ideals coincide with the family \(\{\rann_R(x)\mid x\in R\}\). Thus \(R\) satisfies ACC on principal right annihilators, so \(S\) also satisfies ACC on principal right annihilators. Hence \(S\) is right endonoetherian.

    Now assume \(S\) is also a principal right ideal ring and right \(p\)-Johns. Applying Proposition~\ref{prop:endonoetherian-noetherian} to \(S\), we obtain that \(S\) is right noetherian. By Morita invariance of right noetherianity and assumption (2), the same holds for \(R\).
\end{proof}

\begin{remark}
Corollary~\ref{cor:endonoetherian-morita} clarifies the obstruction exhibited by Shepherdson's example (Example~\ref{ex:shepherdson}). In that construction, $R$ is a domain, hence $\rann_R(a)=0$ for every nonzero $a\in R$. Since $R$ is not a division ring, it has nonzero proper principal right ideals, and none of these can be annihilators of elements of $R$. Thus $R$ is not right $p$-dual and therefore, not right $p$-Johns. Consequently, the standard Morita equivalence between $R$ and $\mathbb{M}_2(R)$ does not satisfy the hypotheses of Corollary~\ref{cor:endonoetherian-morita}. This explains why the endonoetherian property fails to pass to the matrix ring in that example. 

Indeed, Theorem~\ref{thm:morita-princ-ann-acc} shows that the Morita equivalence maps the family of principal right annihilators of $R$ bijectively onto the kernels of endomorphisms of the right $S$-module $P=e_{11}S$, where $S=\mathbb{M}_2(R)$. Since $S_S\cong e_{11}S\oplus e_{22}S$, we have $P\not\cong S_S$. Hence the transported family is not the family of principal right annihilators of $S$. The infinite chain in $S$ arises from endomorphisms of the regular module $S_S$, which have no counterpart in the transported family precisely because the hypothesis $P\cong S_S$ of Theorem~\ref{thm:morita-princ-ann-acc} fails. Thus the obstruction is exactly the failure of the regular-module correspondence, confirming the essentiality and hypotheses of the lemma.
\end{remark}

Finally, we note that the $p$-quasi-Frobenius rings, constructed in this paper, are endonoetherian but exhibit properties that are different from the  noetherian case. It is, therefore, natural to ask the following question.

\begin{question}\label{q:endonoetherian-modules-abelian}
    Let \(R\) be a right endonoetherian ring. Is the full subcategory of \(\mathrm{Mod}\text{-}R\) consisting of endonoetherian right \(R\)-modules abelian? If not, what additional closure conditions must be imposed to obtain an abelian subcategory analogous to \(\mathrm{mod}\text{-}R\) (\(\mathrm{fp}\text{-}R\)) over a right Noetherian (resp.\ right coherent) ring?
\end{question}

 Over a right noetherian ring, the abelianness of \(\mathrm{mod}\text{-}R\) is the key input in Krause's construction of the stable derived category of a noetherian scheme \cite{krause2005stable}, among others. If the corresponding category of endonoetherian modules is abelian, one could try to repeat these constructions for endonoetherian rings, obtaining stable derived categories for non-noetherian rings and potentially extending support-theoretic methods in the spirit of Zou \cite{zou2025support} (though, as that work indicates, such methods are unlikely to extend beyond the weakly noetherian setting, see \cite{zou2025supportreview}  for a brief overview). If the answer is negative, the failure is still informative: it would show where the analogy with the noetherian case breaks down and would indicate that exact categories, abelian envelopes, or a smaller subcategory defined by extra closure conditions are needed. Either way, the answer would clarify which homological constructions genuinely depend on Noetherianity and which survive when noetherianity is replaced by its endomorphism-theoretic counterpart.

\section*{Summary of the Theory}

For the convenience of the reader, we conclude by summarizing the relationships among the classical ring classes and their principal analogues established or discussed in this paper. Green arrows denote implications proved under the hypotheses stated in the paper, while dashed arrows indicate conjectural implications or known non-implications.

\newpage

\begin{center}
\begin{tikzpicture}[scale=1,transform shape,>=latex,box/.style={
draw, rounded corners, minimum width=1.9cm, minimum height=0.8cm, align=center, font=\small}]


\node[box] (QF) at (-6,-1.2) {QF};

\node[box] (SJ) at (-3,-1.2) {Strongly\\Right Johns};

\node[box] (J) at (0,-1.2) {Right\\Johns};

\node[box] (D) at (3,-1.2) {Right\\Dual};

\node[box,fill=green!12] (SP) at (6,-1.2) {Semiperfect};


\draw[->,thick] (QF)--(SJ);
\draw[->,thick] (SJ)--(J);
\draw[->,thick] (J)--(D);
\draw[->,thick] (D)--(SP); 

\draw[->,red,dashed,bend left=35] (SJ.south) to 
    node[below, red, align=center] {
        \fontsize{7}{3}\selectfont 
        Faith--Menal \\  
        \fontsize{7}{3}\selectfont  Conjecture
    } (QF.south);


\node[box,fill=blue!10] (pQF) at (-6,1.6) {{\color{blue}$p$}-QF};

\draw[->,blue,very thick]
(QF)--node[left]{\scriptsize generalizes}(pQF);

\node[box,fill=blue!10] (SpJ) at (-3,1.6) {Strongly\\Right {\color{blue}$p$}-Johns};

\node[box,fill=blue!10] (pJ) at (0,1.6) {Right\\{\color{blue}$p$}-Johns};

\node[box,fill=blue!10] (pD) at (3,1.6) {Right\\{\color{blue}$p$}-Dual};

\node[box,fill=green!12] (PI) at (6,1.6) {Left Principally\\Injective};


\draw[->,thick] (SpJ)--(pJ);
\draw[->,thick] (pJ)--(pD);
\draw[->,thick] (pD)--(PI); 

\draw[<->,purple,thick]
(D)--node[right]{\scriptsize independent}(pD);

\draw[<->,purple,thick]
(SP)--node[right]{\scriptsize independent}(PI); 

\draw[-{Implies}, double, double distance=1pt, red] (pQF) -- 
    node[midway, sloped] {\tikz[thick, red]\draw(-4.5,1.25) -- (-4.5,1.75);} 
    (SpJ);

\draw[{Implies}-{Implies}, double, double distance=1pt,red, bend right=30] (pQF) to 
    node[midway, sloped] {\tikz[thick, red]\draw(-0.1,-0.25) -- (0.1,0.25);} 
    (pD);

\draw[->,green!60!black,bend right=45] (pJ.north) to 
    node[below=0.1cm, blue, align=center] 
    {\fontsize{7}{3}\selectfont 
     \color{red}$p$-Faith--Menal= \\  
     \fontsize{7}{3}\selectfont  \color{red}Two-sided {\color{blue}$p$}-Johns} 
    ([xshift=0.3cm]pQF.north);


\node[box] (CF) at (6,4.4) {Right\\ CF};

\draw[->,thick] (CF)--(PI);

\draw[{Implies}-{Implies}, double, double distance=1pt, red, bend left=30] (pQF.north) to 
    node[midway, sloped] {\tikz[thick]\draw(-0.1,-0.25) -- (0.1,0.25);} 
    node[above=3pt]{} ([yshift=0.1cm]CF.west);

\draw[{Implies}-{Implies}, double, double distance=1pt, red] ([yshift=-0.1cm]CF.west) -- 
    node[midway, sloped] {\tikz[thick, red]\draw(-0.1,-0.25) -- (0.1,0.25);} 
    (pD);

\end{tikzpicture}
\end{center}

\bibliography{ref}
\end{document}